\documentclass[12pt,reqno]{amsart}
\usepackage{amsmath, amsthm, amssymb}
\usepackage[margin=2cm]{geometry}
\usepackage{epsfig}
\usepackage{rawfonts}
\usepackage{enumerate}
\usepackage{graphics}
\usepackage{multirow}
\usepackage{xspace}
\usepackage{graphicx}
\usepackage{pgf,tikz,pgfplots}
\usepackage{mathrsfs}
\usetikzlibrary{arrows}
\usepackage{amsmath}
\usepackage{amsfonts}
\usepackage{amssymb}
\usepackage{amsthm}
\usepackage{graphicx}
\usepackage{booktabs}
\usepackage{caption}
\usepackage{listings}
\usepackage{setspace}
\usepackage[mathscr]{eucal}
\usepackage{pgfplots}
\usepackage{hyperref}
\usepackage{wrapfig}
\usepackage{floatflt,epsfig}
\usepackage{ dsfont }
\usepackage{amscd}
\usepackage{tikz-cd}
\usepackage{fancyhdr}
\usepackage[all]{xy}
\usepackage{latexsym}
\usepackage{amscd}
\usepackage{pifont}
\usepackage{eufrak}
\usepackage{subfig}
\usepackage{easyReview}
\usepackage{todonotes}

\newcommand{\numberset}{\mathbb}

\newcommand{\Z}{\numberset{Z}}

\def\ZZ{{\mathbb Z}}

\newcommand{\lt}{\mathop{\rm in}\nolimits}

\newcommand{\cB}{\mathcal{B}}
\newcommand{\cC}{\mathcal{C}}

\newcommand{\cP}{\mathcal{P}}
\newcommand{\cH}{\mathcal{H}}
\newcommand{\cI}{\mathcal{I}}

\newcommand{\cQ}{\mathcal{Q}}
\newcommand{\cW}{\mathcal{W}}
\newcommand{\cR}{\mathcal{R}}
\newcommand{\cS}{\mathcal{S}}
\newcommand{\cT}{\mathcal{T}}

\newcommand{\rHP}{\mathrm{HP}}

\theoremstyle{plain}

\newtheorem{defn}{Definition}[section]
\newtheorem{prop}[defn]{Proposition}
\newtheorem{thm}[defn]{Theorem}
\newtheorem{discussion}[defn]{Discussion}
\newtheorem{lemma}[defn]{Lemma}

\newtheorem{coro}[defn]{Corollary}
\newtheorem{exa}[defn]{Example}
\newtheorem{rmk}[defn]{Remark}

\newtheorem{stp}[defn]{Set-up}
\newtheorem{qst}[defn]{Question}

\theoremstyle{remark}

\begin{document}

\title[On Cohen-Macaulay non-prime collections of cells]{On Cohen-Macaulay non-prime collections of cells}

\author{CARMELO CISTO}
	\address{Università di Messina, Dipartimento di Scienze Matematiche e Informatiche, Scienze Fisiche e Scienze della Terra\\
		Viale Ferdinando Stagno D'Alcontres 31\\
		98166 Messina, Italy}
	\email{carmelo.cisto@unime.it}
	
 \author{RIZWAN JAHANGIR}
	\address{Sabanci University, Faculty of Engineering and Natural Sciences, Orta Mahalle, Tuzla 34956, Istanbul, Turkey}
	\email{rizwan@sabanciuniv.edu}
	
	\author{FRANCESCO NAVARRA}
	\address{Sabanci University, Faculty of Engineering and Natural Sciences, Orta Mahalle, Tuzla 34956, Istanbul, Turkey}
	\email{francesco.navarra@sabanciuniv.edu}

	

	\begin{abstract} 
    In this paper we investigate Cohen-Macaulayness, Gorensteinness and the Hilbert-Poincar\'e series for some classes of non-prime collections of cells. In particular, we show that all closed path polyominoes are Cohen-Macaulay and we characterize those that are Gorenstein.
        \end{abstract}

	\maketitle
	
	\section*{Introduction}
	
  Combinatorial Commutative Algebra has been an intensive area of research since the pioneering work of R. Stanley (\cite{Stanley2}). In 2012 a new topic emerged in this field due to the work of A. A. Qureshi (\cite{Qureshi}). She established a connection between collections of cells and Commutative Algebra, assigning to every collection $\cP$ of cells the ideal of the inner 2-minors of $\cP$ in a suitable polynomial ring $S_{\cP}$. This ideal $I_{\mathcal{P}}$ is called the \textit{inner 2-minor} ideal of $\cP$ and $K[\mathcal{P}]=S_{\cP}/I_{\mathcal{P}}$ is said to be the \textit{coordinate ring} of $\cP$. In particular, if $\cP$ is a polyomino, that is a collection of cells where the cells are joined edge by edge, then $I_\cP$ is called the \textit{polyomino ideal} of $\cP$. The aim of the research is to study the main algebraic properties of $K[\cP]$ depending on the geometry of $\cP$. The most investigated include primality, Cohen-Macaulayness and Gorensteinness. As usual, we say that a collection of cells is prime (respectively, Cohen-Macaulay or Gorenstein) if the related binomial ideal is prime (respectively, Cohen-Macaulay or Gorenstein).
 Regarding the classification of the primality for polyominoes, it is proven in \cite{Simple equivalent balanced, def balanced, Simple are prime} that if $\cP$ is simple, which means that $\cP$ has no hole, then $K[\cP]$ is a normal Cohen-Macaulay domain. Nowadays, the study is devoted to non-simple polyominoes but a complete characterization of primality, Cohen-Macaulay and Gorenstein properties is still unknown, despite the efforts of many mathematicians (see \cite{Andrei, Cisto_Navarra_closed_path, Cisto_Navarra_weakly, Cisto_Navarra_Hilbert_series, Simple equivalent balanced, def balanced, Not simple with localization, Trento,Trento2, Parallelogram Hilbert series, Simple are prime, Trento3,Shikama}). Other very interesting combinatorial problems on polyomino ideals have been investigated as shown in \cite{Cisto_Navarra_Veer, Dinu_Navarra_Konig, Herzog rioluzioni lineari, Hibi - Herzog Konig type polyomino, Kummini CD}. 
 
 Cohen-Macaulayness and Goresteiness have so far been discussed only for prime collections of cells. In all the papers where Cohen-Macaulayness is studied (\cite{Cisto_Navarra_CM_closed_path, Not simple with localization, Frame, Trento2, Simple are prime}), the demonstrative strategy is based on two well-known theorems of Sturmfels (\cite[Chapter 13]{Sturm}) and Hochster (\cite[Theorem 6.3.5]{Bruns_Herzog}), which state that a toric ring whose defining ideal admits a squarefree initial ideal is a normal Cohen-Macaulay ring. For Goresteiness, a well-known result of Stanley (\cite{Stanley}) allows to characterize the prime polyominoes which are Gorenstein, by checking if the $h$-polynomial is palindromic. In such a context, the so-called rook polynomial plays a crucial role. The rook polynomial is a polynomial whose $k$-th coefficient is the number of ways of placing $k$ non-attacking rooks in a polyomino. The maximum number of non-attacking rooks is called the rook number of $\cP$. The connection of the Castelnuovo-Mumford regularity of $K[\cP]$ with the rook number was shown in \cite{L-convessi} for the first time. So far, the study of the $h$-polynomial in terms of the rook polynomial has produced a relevant number of articles, including \cite{Cisto_Navarra_Hilbert_series, Dinu_Navarra_grid, Kummini rook polynomial, Frame, Parallelogram Hilbert series, Trento3}, and it could help to characterize the prime polyominoes that are Gorenstein. Indeed, if the $h$-polyonomial of a prime polyomino $\cP$ is the rook polynomial, then we can check the palindromicity counting the number of suitable configurations of non-attacking rooks in $\cP$. On the other hand, discussing non-prime polyominoes which are Gorenstein seems very hard and such classes have not appeared so far.
 
 In this context a natural question arises: what can we say about the Cohen-Macaulay and the Gorenstein properties for non-prime collections of cells? Motivated by this question, we begin to investigate these properties for non-prime collections of cells. From \cite[Corollary 3.6]{Trento} we know that if a collection of cells contains the so-called zig-zag walks (see \cite[Definition 3.2]{Trento}), then the related binomial ideal is not prime. This suggests to consider as a first step a class of non-prime collections of cells made up of zig-zag walks, which we called \textit{zig-zag collections}. As conjectured in \cite[Conjecture 4.6]{Trento}, every non-prime collection of cells should contain a zig-zag collection, so the latter could be useful for gaining information on non-prime collections of cells, by applying a decomposition of them into suitable zig-zag collections. This method is used to study the before mentioned properties for the \emph{closed paths}, a kind of non-simple polyominoes introduced for the first time in \cite{Cisto_Navarra_closed_path}, in the non-prime case. In particular, for this class of polyominoes, we state the following result at the end of this paper: 
 
 \vspace{0.2cm}
\textbf{Theorem \ref{thm:summary-closed-path}} {\em
    Let $\cP$ be a closed path polyomino. Then:
\begin{enumerate}
    \item  $K[\cP]$ is Cohen-Macaulay with Krull dimension $|V(\cP)|-|\cP|$; moreover, if $\cP$ does not contain any zig-zag walk, then $K[\cP]$ is also a normal domain. 
    \item The $h$-polyonomial of $K[\cP]$ is the rook polynomial of $\cP$ and $\mathrm{reg}(K[\cP])=r(\cP)$.
 \item $K[\cP]$ is Gorenstein if and only if $\cP$ consists of maximal blocks of rank three.
\end{enumerate}
}
 \vspace{0.2cm}
The article is structured as follows. Section 1 is devoted to provide some basic definitions and properties of collections of cells and the associated coordinate rings. In Section 2 we prove that the $h$-polynomial of a simple collection of cells is equal to the product of the $h$-polynomial of each connected component (Theorem \ref{thm:product-h-poly}), exploiting the monomial order in \cite{Ohsug-Hibi_koszul} and \cite[Theorem 3.3]{Cisto_Navarra_weakly}. As a direct consequence, we can generalize \cite[Theorem 3.13]{Trento}, obtaining in Corollary \ref{cor:hilbert-series-thin-collections} that the $h$-polynomial of a simple thin collection of cells coincides with the rook-polynomial. In Section 3 we introduce a new class of non-prime collections of cells, named \emph{zig-zag collections}. Taking advantage of some recent results shown in \cite{conca4}, we show that their coordinate rings are Cohen-Macaulay (Theorem~\ref{zig-collection}), we compute the Hilbert-Poincar\'e series in terms of the switching rook polynomial (Corollary~\ref{Coro: For zig-zag coll, h is the switching rook pol}) and we eventually discuss the Gorenstein property (Corollary~ \ref{cor:gor-zig}). The results achieved in the previous sections are collected in Section 4 to prove Theorem~\ref{thm:summary-closed-path}, which is our main result.

\section{Collections of cells and their associated coordinate rings}

\noindent This section is devoted to introduce several preliminaries. In particular, we give the definitions and the notations related to collections of cells and the related binomial ideals (see \cite{Qureshi}) and we mention several well-known results in commutative algebra. We firstly point out that, in the whole paper, for each positive integer $n$ we use the standard notation $[n]$ to denote the set $\{1,\ldots,n\}$.

 \subsection{Collections of cells.} Let $(i,j),(k,l)\in \ZZ^2$. We say that $(i,j)\leq(k,l)$ if $i\leq k$ and $j\leq l$. Consider $a=(i,j)$ and $b=(k,l)$ in $\ZZ^2$ with $a\leq b$. The set $[a,b]=\{(m,n)\in \ZZ^2: i\leq m\leq k,\ j\leq n\leq l \}$ is called an \textit{interval} of $\ZZ^2$. 
If $i< k$ and $j<l$, then we say that $[a,b]$ is a \textit{proper} interval. In this case, we call $a$ and $b$ the \textit{diagonal corners} of $[a,b]$, and $c=(i,l)$ and $d=(k,j)$ the \textit{anti-diagonal corners} of $[a,b]$. If $j=l$ (or $i=k$), then $a$ and $b$ are in \textit{horizontal} (or \textit{vertical}) \textit{position}. A proper interval $C=[a,b]$ with $b=a+(1,1)$ is called a \textit{cell} of $\ZZ^2$; moreover, the elements $a$, $b$, $c$ and $d$ are called respectively the \textit{lower left}, \textit{upper right}, \textit{upper left} and \textit{lower right} \textit{corners} of $C$. The set of vertices of $C$ is $V(C)=\{a,b,c,d\}$ and the set of edges of $C$ is $E(C)=\{\{a,c\},\{c,b\},\{b,d\},\{a,d\}\}$. Let $\cS$ be a non-empty collection of cells in $\Z^2$. Then $V(\cS)=\bigcup_{C\in \cS}V(C)$ and $E(\cS)=\bigcup_{C\in \cS}E(C)$, while the rank of $\cS$ is the number of cells that belong to $\cS$. If $C$ and $D$ are two distinct cells of $\cS$, then a \textit{walk} from $C$ to $D$ in $\cS$ is a sequence $\cC:C=C_1,\dots,C_m=D$ of cells of $\ZZ^2$ such that $C_i \cap C_{i+1}$ is an edge of $C_i$ and $C_{i+1}$ for $i=1,\dots,m-1$. Moreover, if $C_i \neq C_j$ for all $i\neq j$, then $\cC$ is called a \textit{path} from $C$ to $D$. 
 \noindent We say that $C$ and $D$ are \textit{connected} in $\cS$ if there exists a path of cells in $\cS$ from $C$ to $D$. A \textit{polyomino} $\cP$ is a non-empty, finite collection of cells in $\Z^2$ where any two cells of $\cP$ are connected in $\cP$ (see Figure \ref{Figure: Polyomino introduction} (a)) Let $\cP$ be a non-empty, finite collection of cells in $\Z^2$, $\cP$ is called \textit{weakly connected} if for any two cells $C$ and $D$ in $\cP$ there exists a sequence of cells $\cC: C=C_1,\dots,C_m=D$ of $\cP$ such that $V(C_i)\cap V(C_{i+1}) \neq \emptyset$ for all $i=1,\dots,m-1$. A subset of cells $\cP'$ of $\cP$ is called a \textit{connected component} of $\cP$ if $\cP'$ is a polyomino and it is maximal with respect to the set inclusion, that is, if $A\in \cP\setminus \cP'$ then $\cP'\cup \{A\}$ is not a polyomino. For instance, see Figure \ref{Figure: Polyomino introduction} (b). Observe trivially that every polyomino is a weakly connected collection of cells.

\begin{figure}[h]
\centering
\subfloat[]{\includegraphics[scale=0.5]{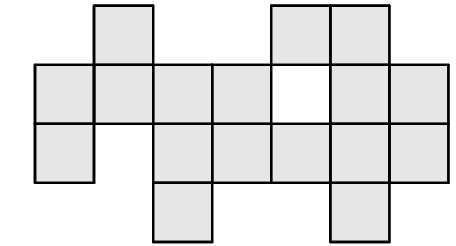}\label{non-simple_one}}\
\subfloat[]{\includegraphics[scale=0.5]{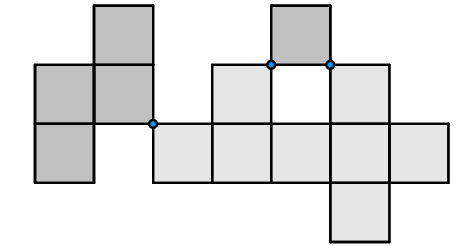}\label{non-simple_two}}
\subfloat[]{\includegraphics[scale=0.55]{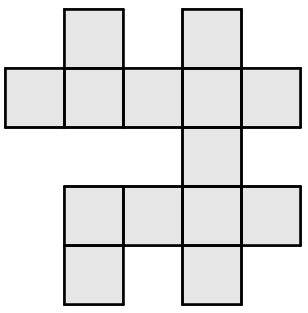}\label{leaf1}}\
\subfloat[]{\includegraphics[scale=0.55]{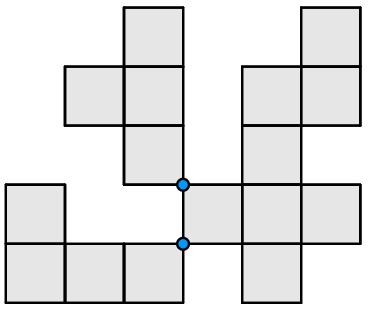}\label{leaf2}}
\caption{A non-simple polyomino, a non-simple collection of cells, a simple thin polyomino and a simple thin collection of cells.}
\label{Figure: Polyomino introduction}
\end{figure}
	
\noindent A \textit{sub-polyomino} of $\cP$ is a polyomino whose cells belong to $\cP$. A weakly connected collection of cells is said to be \textit{thin} if it does not contain the square tetromino (that is, the square obtained as a union of four distinct cells); see Figures \ref{leaf1} and \ref{leaf2}. We say that $\cP$ is \textit{simple} if for any two cells $C$ and $D$ not in $\cP$ there exists a path of cells not in $\cP$ from $C$ to $D$. A finite collection of cells $\cH$ not in $\cP$ is an \textit{hole} of $\cP$ if any two cells of $\cH$ are connected in $\cH$ and $\cH$ is maximal with respect to set inclusion. For example, the collections of cells in Figures \ref{non-simple_one}, \ref{non-simple_two} and \ref{Fig:non-simple} are not simple. Each hole of $\cP$ is a simple polyomino, and $\cP$ is simple if and only if it has no hole. An interval $[a,b]$ with $a=(i,j)$, $b=(k,j)$ and $i<k$ is called a \textit{horizontal edge interval} of $\cP$ if the sets $\{(\ell,j),(\ell+1,j)\}$ are edges of cells of $\cP$ for all $\ell=i,\dots,k-1$. In addition, if $\{(i-1,j),(i,j)\}$ and $\{(k,j),(k+1,j)\}$ do not belong to $E(\cP)$, then $[a,b]$ is called a \textit{maximal} horizontal edge interval of $\cP$. We define similarly a \textit{vertical edge interval} and a \textit{maximal} vertical edge interval.\\
\noindent	Following \cite{Trento}, we call a \textit{zig-zag walk} of $\cP$ a sequence $\cW:I_1,\dots,I_\ell$ of distinct inner intervals of $\cP$ where, for all $i=1,\dots,\ell$, the interval $I_i$ has either diagonal corners $v_i$, $z_i$ and anti-diagonal corners $u_i$, $v_{i+1}$ or anti-diagonal corners $v_i$, $z_i$ and diagonal corners $u_i$, $v_{i+1}$, such that:
	\begin{enumerate}
		\item $I_1\cap I_\ell=\{v_1=v_{\ell+1}\}$ and $I_i\cap I_{i+1}=\{v_{i+1}\}$, for all $i=1,\dots,\ell-1$;
		\item $v_i$ and $v_{i+1}$ are on the same edge interval of $\cP$, for all $i=1,\dots,\ell$;
		\item for all $i,j\in \{1,\dots,\ell\}$ with $i\neq j$, there exists no inner interval $J$ of $\cP$ such that $z_i$, $z_j$ belong to $J$.
	\end{enumerate}
 
	\begin{figure}[h]
	\centering
	\subfloat[An example of a zig-zag walk of $\cP$.]{\includegraphics[scale=0.65]{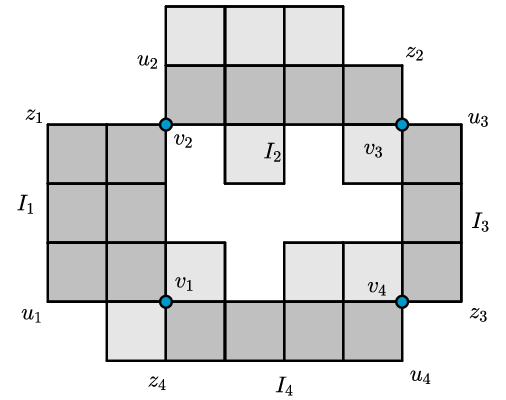}
    \label{Fig: zig zag}} \qquad
        \subfloat[An example of a $6$-rook configuration in $\cP$.]{\includegraphics[scale=0.8]{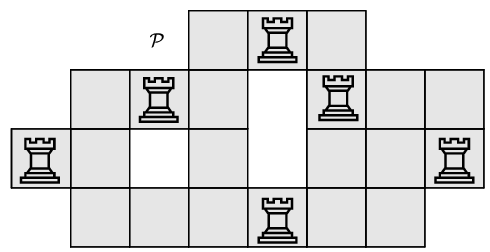}
        \label{Fig:rook}}
	\caption{Non-simple polyominoes.}
        \label{Fig:non-simple}
        \end{figure}

\noindent See Figure \ref{Fig: zig zag} for an example. 

Consider two cells $A$ and $B$ of $\Z^2$ with $a=(i,j)$ and $b=(k,l)$ as the lower left corners of $A$ and $B$ with $a\leq b$. A \textit{cell interval} $[A,B]$ is the set of the cells of $\Z^2$ with lower left corner $(r,s)$ such that $i\leqslant r\leqslant k$ and $j\leqslant s\leqslant l$. If $(i,j)$ and $(k,l)$ are in horizontal (or vertical) position, we say that the cells $A$ and $B$ are in \textit{horizontal} (or \textit{vertical}) \textit{position}.\\
Let $\cP$ be a polyomino. Consider  two cells $A$ and $B$ of $\cP$ in a vertical or horizontal position. The cell interval $[A,B]$, containing $n>1$ cells, is called a \textit{block of $\cP$ of rank n} if all cells of $[A,B]$ belong to $\cP$. The cells $A$ and $B$ are called \textit{extremal} cells of $[A,B]$. We set also $[A,B[=[A,B]\setminus\{B\}$ and $]A,B[=[A,B]\setminus\{A,B\}$. Moreover, a block $\cB$ of $\cP$ is \textit{maximal} if there does not exist any block of $\cP$ which contains properly $\cB$. It is clear that an interval of $\ZZ^2$ identifies a cell interval of $\ZZ^2$ and vice versa, hence we can associate to an interval $I$ of $\ZZ^2$ the corresponding cell interval denoted by $\cP_{I}$. A proper interval $[a,b]$ is called an \textit{inner interval} of $\cP$ if all cells of $\cP_{[a,b]}$ belong to $\cP$. We denote by $\cI(\cP)$ the set of all inner intervals of $\cP$.

\subsection{Switching rook polynomial.}

Let $\cP$ be a collection of cells. Two rooks in $\cP$ are in \textit{non-attacking position} if they do not belong to the same vertical or horizontal cell interval contained in $\cP$. A \textit{$k$-rook configuration} in $\cP$ is a configuration of $k$ rooks which are arranged in $\cP$ in non-attacking positions. Figure~\ref{Fig:rook} shows a 6-rook configuration.
	
	\noindent The rook number $r(\cP)$ is the maximum number of rooks that can be arranged in $\cP$ in non-attacking positions. For all $k\in \{0,\ldots,r(\cP)\}$, we denote by $\cR_k$ the set of all $k$-rook configurations in $\cP$, using the conventional setting $R_0=\{\emptyset\}$.  We also set $r_k=\vert \cR_k\vert $ for all $k\in \{0,\dots,r(\cP)\}$. The \textit{rook-polynomial} of $\cP$ is the polynomial in $\mathbb{Z}[t]$ defined as $r_{\cP}(t)=\sum_{k=0}^{r(\cP)}r_kt^k$. Moreover, the set $\cR=\cR_0\cup \cR_1\dots \cup\cR_{r(\cP)}$ is a simplicial complex, called \textit{rook complex}. We call a \emph{rectangle} of $\cP$ the set of all cells belonging to an inner interval of $\cP$ and we say that two cells $A, B$, with left lower corners $(i,j), (h,k)$ respectively, are in diagonal position if $(i+1,j+1)\leq (h,k)$, while they are in anti-diagonal) if $(i+1,j+1)$ and $(h,k)$ are not comparable with respect to $\prec$. Two non-attacking rooks in $\cP$ are \textit{switching} rooks if they are placed in diagonal or anti-diagonal cells of a rectangle of $\cP$. In such a case we say that the rooks are in a diagonal or anti-diagonal position, respectively. Fix $k\in \{0,\dots, r_{\cP}\}$. Let $F\in \cR_k$ and $R_1$ and $R_2$ be two switching rooks in $F$ in diagonal (resp. anti-diagonal) position. Let $R_1'$ and $R_2'$ be the rooks in anti-diagonal (resp. diagonal) cells in $R$. Then the set $\left(F\backslash \{R_1, R_2\}\right) \cup \{R_1', R_2'\}$ belongs to $\cR_k$. The operation of replacing $R_1$ and $R_2$ by $R_1'$ and $R_2'$ is called \textit{switch of $R_1$ and $R_2$}. This induces the following equivalence relation $\sim$ on $\cR_k$: let $F_1, F_2 \in \cR_k$, then $F_1\sim F_2$ if $F_2$ can be obtained from $F_1$ after some switches. We define the quotient set $\tilde{\cR}_k = \cR_k/\sim$. We set $\tilde{r}_k=\vert \tilde{\cR}_k\vert $ for all $k\in [r(\cP)]$. The \textit{switching rook-polynomial} of $\cP$ is the polynomial in $\mathbb{Z}[t]$ defined as $\tilde{r}_{\cP}(t)=\sum_{k=0}^{r(\cP)}\tilde{r}_kt^k$. Observe that if $\cP$ is a thin polyomino then $\tilde{r}_{\cP}(t)=r_{\cP}(t).$ Moreover, let $\cP$ be a collection of cells. Consider an element $\cC$ of $\tilde{R}_k$ and we assume that all rooks in $\cC$ are in diagonal position. We say that $\cC$ is a \textit{canonical configuration} in $\cP$ of $k$ rooks, so $\tilde{r}_k$ is the number of canonical positions in $\cP$ of $k$ rooks.

\begin{lemma}\label{Lemma: prod of rook polynomial}
    Let $\cP$ be a collection of cells consisting of $n$ connected components $\cP_1,\dots,\cP_n$. 
    Denote by $\tilde{r}_{\cP_i}(t)$ the switching rook polynomial of $\cP_i$. Then $\prod_{i=1}^n\tilde{r}_{\cP_i}(t)$ is the switching rook polynomial of $\cP$.
\end{lemma}

\begin{proof}
    It is enough to prove the claim when $\cP$ is made up of two connected components $\cP_1$ and $\cP_2$. Indeed, if the latter is proved, then $\tilde{r}_1(t)\tilde{r}_2(t)$ is the switching rook polynomial of $\cP_1\cup\cP_2$. Now, we consider $(\cP_1\cup \cP_2)\cup \cP_3$, so we have that $(\tilde{r}_1(t)\tilde{r}_2(t))\tilde{r}_3(t)$ is the switching rook polynomial of $(\cP_1\cup \cP_2)\cup \cP_3$. We can continue these arguments until $\cP_n$, getting that $\prod_{i=1}^{n} \tilde{r}_i(t)$ is the switching rook polynomial of $\cP$.\\  
    Assume that $\cP$ consists of two connected components $\cP_1$ and $\cP_2$. Denote by $\tilde{r}(t)$ the switching rook polynomial of $\cP$ and we set $\tilde{r}_1(t)=\sum_{i=0}^{r(\cP_1)}r_{i}^{(1)}t^i$ and $\tilde{r}_2(t)=\sum_{j=0}^{r(\cP_2)}r_{j}^{(2)}t^j.$ Hence
    $$\tilde{r}_1(t)\tilde{r}_2(t)=\Bigg(\sum_{i=0}^{r(\cP_1)}r_{i}^{(1)}t^j\Bigg)\Bigg(\sum_{j=0}^{r(\cP_2)}r_{j}^{(2)}t^j\Bigg)=\sum_{k=0}^{r(\cP_1)+r(\cP_2)}c_{k}t^k.$$
    where $c_k=r_{0}^{(1)}r_{k}^{(2)}+r_{1}^{(1)}r_{k-1}^{(2)}+r_{2}^{(1)}r_{k-2}^{(2)}+\dots+r_{k}^{(1)}r_{0}^{(2)}$. From the structure of $\cP$, observe that two rooks placed in different connected components are surely in non-attacking position. So, the number of the canonical configurations of $k$-rooks in $\cP$ is obtained considering the sum of the products of the number of the canonical configurations of $k_1$-rooks in $\cP_1$ and that one of the canonical configurations of $k_2$-rooks in $\cP_2$, for all non-negative integers $k_1,k_2$ such that $k_1+k_2=k$. Hence, it follows that $c_k$ represents the number of the canonical configurations of $k$-rooks in $\cP$, which is the $k$-th coefficient of $\tilde{r}(t)$. Therefore $\tilde{r}_1(t)\tilde{r}_2(t)$ is the switching rook polynomial of $\cP$.
\end{proof}

\subsection{The coordinate ring attached to a collection of cells}

Following \cite{Qureshi}, we define a binomial ideal attached to a collection of cells. Let $\cP$ be a collection of cells. Set  $S_\cP=K[x_v| v\in V(\cP)]$, where $K$ is a field. If $[a,b]$ is an inner interval of $\cP$, with $a$,$b$ and $c$,$d$ respectively diagonal and anti-diagonal corners, then the binomial $x_ax_b-x_cx_d$ is called an \textit{inner 2-minor} of $\cP$. We defined the binomial ideal $I_{\cP}$ in $S_\cP$ as the ideal generated by all the inner 2-minors of $\cP$ and we call it the \textit{inner 2-minors ideal} of $\cP$. If $\cP$ is a polyomino, then $I_\cP$ is simply said the \textit{polyomino ideal} of $\cP$. We set also $K[\cP] = S_\cP/I_{\cP}$, which is the \textit{coordinate ring} of $\cP$. We denote by $G(\cP)$ the set of the generators of $I_{\cP}$. For a generalization of collections of cells and the related algebras look at \cite{Cisto_Navarra_Veer}.

 \begin{rmk}\rm 
     Let $\cP$ be a collection of cells and assume that it is written as a disjoint union of the collection of cells $\cP_1,\dots ,\cP_r$ such that $V (\cP)$ is the disjoint union of vertex sets $V(\cP_j)$, with $j\in [r]$. Then $K[\cP]\cong \bigotimes_{i=1}^r K[\cP_j]$. It follows that $K[\cP]$ is a normal Cohen–Macaulay domain if and only if $K[\cP_j]$ is a normal Cohen–Macaulay domain for all $j\in[r]$. Hence it is not restrictive to consider just weakly connected collections of cells. For this reason, whenever the term ``collection of cells'' is introduced in the paper,  we mean a weakly connected collection of cells.
 \end{rmk}

\noindent For utility, we state the following result about the Cohen-Macaulay property of the simple collection of cells, that in the literature we found stated in this form only in the case $\cP$ is polyomino (see \cite[Theorem 2.1]{Simple equivalent balanced} and \cite[Corollary 3.3]{def balanced}). For the sake of completeness, we also provide its proof.

\begin{prop}\label{prop:simple_collection-cells-areCM}
    Let $\cP$ be a simple collection of cells. Then $K[\cP]$ is a normal Cohen-Macaulay domain of dimension $|V(\cP)|-|\cP|$.
\end{prop}
\begin{proof}
    By \cite[Theorem 3.3]{Cisto_Navarra_weakly} we know that $I_\cP=J_\cP$ where $J_\cP$ is the toric ring of a weakly chordal bipartite graph. In particular $I_\cP$ is prime and, by \cite{Ohsug-Hibi_koszul}, $I_\cP$ has a squarefree quadratic Gr\"obner basis. A toric ring whose toric ideal admits a squarefree initial ideal is normal, see \cite[Corollary 4.26]{binomial ideals} and by a theorem of Hochster (see \cite[Theorem 6.3.5]{Bruns_Herzog}), a normal toric ring is Cohen–Macaulay. Finally, by \cite[Theorem 3.5]{Qureshi} and \cite[Corollary 3.6]{Qureshi} we know that $I_\cP=I_\Lambda$ for a saturated lattice $\Lambda$ such that $\operatorname{rank}_\mathbb{Z}(\Lambda)=|\cP|$. Hence $\operatorname{height}(I_\cP)=|\cP|$ (see \cite[Problem 3.12]{binomial ideals}) and so $\dim K[\cP]=|V(\cP)|-|\cP|$ (see also \cite[Corollary 3.1.7]{Villareal}).
\end{proof}

\noindent In accordance to \cite{Cisto_Navarra_closed_path}, we recall the definition of a \textit{closed path polyomino}, and the configuration of cells characterizing its primality. 

\begin{defn}
We say that a polyomino $\cP$ is a \textit{closed path} if it is a sequence of cells $A_1,\dots,A_n, A_{n+1}$, $n>5$, such that:
\begin{enumerate}
	\item $A_1=A_{n+1}$;
	\item $A_i\cap A_{i+1}$ is a common edge, for all $i=1,\dots,n$;
	\item $A_i\neq A_j$, for all $i\neq j$ and $i,j\in \{1,\dots,n\}$;
	\item For all $i\in\{1,\dots,n\}$ and for all $j\notin\{i-2,i-1,i,i+1,i+2\}$ then $V(A_i)\cap V(A_j)=\emptyset$, where $A_{-1}=A_{n-1}$, $A_0=A_n$, $A_{n+1}=A_1$ and $A_{n+2}=A_2$. 
\end{enumerate}
\end{defn}

\begin{figure}[h!]
	\centering
	\subfloat[]{\includegraphics[scale=0.6]{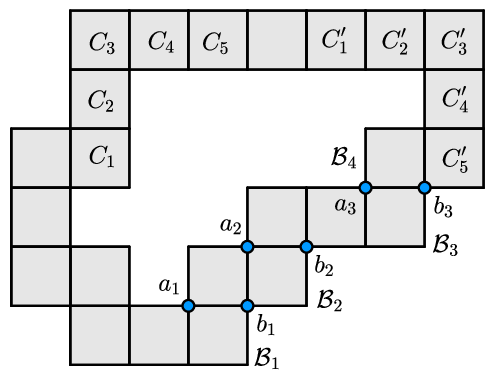} \label{Figura:L conf + Ladder}}\qquad
 \subfloat[]{\includegraphics[scale=0.4]{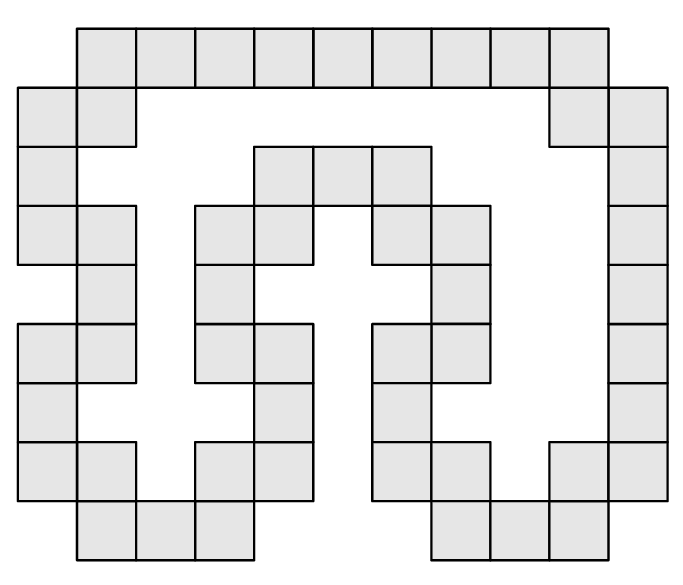}
 \label{Figure: example closed path with zig-zag Big}
 }\qquad
\subfloat[]{\includegraphics[scale=0.5]{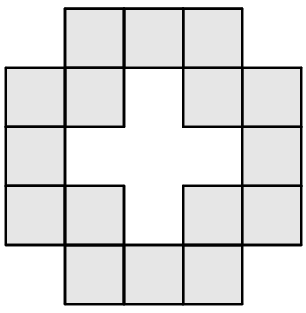}
   \label{Figure: the smallest closed path with zig-zag}}
  \caption{Examples of closed paths.}
\end{figure}
 \noindent A path of five cells $C_1, C_2, C_3, C_4, C_5$ of $\cP$ is called an \textit{L-configuration} if the two sequences $C_1, C_2, C_3$ and $C_3, C_4, C_5$ go in two orthogonal directions. A set $\cB=\{\cB_i\}_{i=1,\dots,n}$ of maximal horizontal (or vertical) blocks of rank at least two, with $V(\cB_i)\cap V(\cB_{i+1})=\{a_i,b_i\}$ and $a_i\neq b_i$ for all $i=1,\dots,n-1$, is called a \textit{ladder of $n$ steps} if $[a_i,b_i]$ is not on the same edge interval of $[a_{i+1},b_{i+1}]$ for all $i=1,\dots,n-2$. For instance, in Figure \ref{Figura:L conf + Ladder} there is a closed path having two $L$-configurations and a ladder of four steps. It is proved in \cite[Section 6]{Cisto_Navarra_closed_path} that a closed path has no zig-zag walks if and only if it contains an $L$-configuration or a ladder of at least three steps. Furthermore, in \cite[Theorem 6.2]{Cisto_Navarra_closed_path} the authors characterize the primality of a closed path $\cP$ proving that $I_{\cP}$ is prime if and only if $\cP$ does not contain any zig-zag walk.

	\section{Hilbert-Poincar\'e series of simple collections of cells}
    
\noindent In this section we present some results on simple collections of cells. From \cite[Theorem 3.3]{Cisto_Navarra_CM_closed_path} and \cite{Ohsug-Hibi_koszul} we know that, if $\cP$ is a simple collection of cells, the set of generators of $I_{\cP}$ forms the reduced Gr\"obner basis of $I_{\cP}$ with respect to a suitable monomial order. The monomial order provided in \cite{Ohsug-Hibi_koszul} is fundamental for our purposes because it allows us to write $S_{\cP}/\lt(I_{\cP})$ as $K$-tensor product of suitable $K$-algebras coming from the connected components. \\
About some notations we use throughout the paper, if $S$ is a polynomial ring, $\prec$ is a monomial order on it and $f\in S$ is a polynomial, we denote by $\operatorname{in}_\prec(f)$, or $\operatorname{in}(f)$ if the order is clear in the context, the initial monomial of $f$ with respect to $\prec$. As usual, if $S=K[x_1,\ldots,x_n]$ and $m$ is a monomial, then $\operatorname{supp}(m)=\{x_i\in S\mid x_i \text{ divides } m\} $. We start proving the following Lemma.
  
\begin{lemma}\label{lemma:s-poly-simple-weak}
	Let $\cP$ be a collection of cells and $[a,b]$ and $[\alpha,\beta]$ be two inner intervals with $\alpha=b$ (see Figure~\ref{img_intervalli8-a}). Let $\prec$ be a monomial order on $S_\cP$ and $f_{a,b}$ 
 and $f_{\alpha,\beta}$ be the generators of $I_{\cP}$ attached to $[a,b]$ and $[\alpha,\beta]$ respectively. Assume that $S(f_{a,b},f_{\alpha,\beta})$ reduces to 0 modulo the set of generators of $I_{\cP}$ with respect to $\prec$. Suppose $[c,\gamma]$ and $[d,\delta]$ are not inner intervals of $\cP$, then $\gcd(\lt(f_{a,b}),\lt(f_{\alpha,\beta}))= 1$. 
   	
\begin{figure}[h]
	\centering
	\subfloat[\label{img_intervalli8-a}]{\includegraphics[scale=0.6]{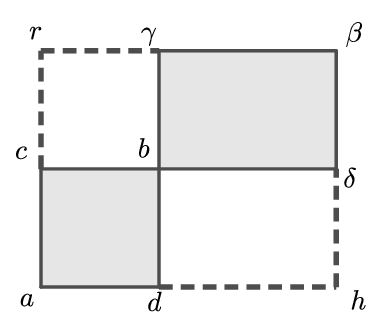}}\qquad\qquad 
	\subfloat[\label{img_intervalli8-b}]{\includegraphics[scale=0.6]{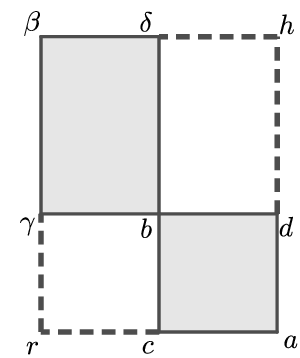}}
	\caption{Arrangements of intervals.}
	\label{img_intervalli8}
\end{figure}
   \noindent The same holds considering Figure~\ref{img_intervalli8-b}, that is, if $S(f_{\gamma,\delta},f_{c,d}))$ reduces to $0$ and  $[r,b]$ and $[b,h]$ are not inner intervals, then $\gcd(\lt(f_{c,d}),\lt(f_{\gamma,\delta}))= 1$.
	\label{intervalli8}
\end{lemma}

\begin{proof}
     From Figure~\ref{img_intervalli8-a} it is clear that $f_{a,b} = x_ax_b - x_cx_d$ and $f_{\alpha, \beta} = x_{\alpha}x_{\beta} - x_{\gamma} x_{\delta}$. Suppose that $\gcd(\lt(f_{a,b}),\lt(f_{\alpha,\beta})) \neq 1$. Then $\lt(f_{a,b}) = x_ax_b$ and $\lt(f_{\alpha, \beta}) = x_{\alpha}x_{\beta} = x_{b}x_{\beta}$. This leads that $S(f_{a,b},f_{\alpha,\beta}) = x_ax_{\gamma}x_{\delta} - x_{\beta}x_{c}x_{d} $. Assume that $\lt(S(f_{a,b},f_{\alpha,\beta}))=x_ax_{\gamma}x_{\delta}$.  Since $\lt(f_{\alpha, \beta}) = x_{b}x_{\beta}$, if $S(f_{a,b},f_{\alpha,\beta}))$ reduces to $0$ then $x_a x_\gamma - x_r x_d$ is an inner 2-minor having initial monomial $x_a x_\gamma$, or $x_a x_\delta- x_c x_h$ is an inner 2-minor with initial monomial $x_a x_\delta$. Hence, we obtain $[c,\gamma]$ or $[d,\delta]$ is an inner interval, a contradiction. Assume that $\lt(S(f_{a,b},f_{\alpha,\beta}))=x_{\beta}x_c x_d$.  Since $\lt(f_{a, b}) = x_{a}x_b$, if $S(f_{a,b},f_{\alpha,\beta}))$ reduces to $0$ then $x_c x_\beta - x_r x_\delta$ is an inner 2-minor having initial monomial $x_c x_\beta$, or $x_d x_\beta- x_\gamma x_h$ is an inner 2-minor with initial monomial $x_d x_\beta$. Hence, we obtain $[c,\gamma]$ or $[d,\delta]$ as an inner interval, a contradiction. The last claim, for Figure~\ref{img_intervalli8-b}, holds using the same arguments. 
\end{proof}

Let $R$ be a graded $K$-algebra and $I$ be a homogeneous ideal of $R$. The formal power series $\rHP_{R/I}(t)=\sum_{k\in\mathbb{N}}\dim_{K} (R/I)_kt^k$ is called the \textit{Hilbert-Poincar\'{e} series} of $R/I$. It is known by Hilbert-Serre theorem that there exists a polynomial $h(t)\in \mathbb{Z}[t]$ with $h(1)\neq0$ such that $\rHP_{R/I}(t)=\frac{h(t)}{(1-t)^d}$, where $d$ is the Krull dimension of $R/I$. The polynomial $h(t)$ is called \textit{h-polynomial} of $R/I$. Moreover, if $R/I$ is Cohen-Macaulay then $\mathrm{reg}(R/I)=\deg h(t)$. Recall also that if $S=K[x_1,\ldots,x_n]$ then $\rHP_{S}(t)=\frac{1}{(1-t)^n}$. In some parts of the paper, if $K[\cP]$ is the coordinate ring of a collection of cells, we use the notation $h_{K[\cP]}$ for the $h$-polyonimial of the Hilbert-Poincar\'{e} series of $K[\cP]$. Once we set the previous notation, let us prove the main result of this section.

\begin{thm}\label{thm:product-h-poly}
Let $\cP$ be a simple collection of cells consisting of $n$ connected components $\cP_{1}, \ldots, \cP_{n}$. Let $h_{K[\cP_{i}]}(t)$ be the $h$-polynomial of $K[\cP_i]$ for each $i=1,\ldots, n$, then $ \prod_{i=1}^{n}h_{K[\cP_{i}]}(t)$ is the $h$-polynomial of $K[\cP]$.
\end{thm}
\begin{proof}
 
Let $\cP$ be a simple collection of cells. According to references \cite[Theorem 3.3]{Cisto_Navarra_weakly} and \cite{Ohsug-Hibi_koszul}, $I_\cP$ possesses a quadratic Gr\"obner basis with respect to a certain monomial order. We fix such a monomial order, then for any pair of inner 2-minors $f$ and $g$, their $S$-polynomial $S(f,g)$ reduces to 0. 
Suppose that $f$ and $g$ are inner 2-minors of $\cP_i$ and $\cP_j$ respectively, where $i \neq j$. Now we have two cases: when $f$ and $g$ are binomials associated with inner intervals that are disjoint or they share a vertex. For the first, we have $\gcd(\lt(f),\lt(g)=1$ trivially. For the second case, since the inner 2-minors $f$ and $g$ are of the form as shown in Figure~\ref{img_intervalli8} and $S(f,g)$ reduces to $0$, so by Lemma~\ref{lemma:s-poly-simple-weak} we obtain $\gcd(\lt(f), \lt(g))=1$ as well. So, for all $i,j\in [n]$ with $i\neq j$ and for all $f,g$ inner 2-minors of $I_\cP$ such that $f$ is related to an inner interval of $\cP_i$ and $g$ is related to an inner interval of $\cP_j$, we have $\gcd(\lt(f),\lt(g))=1$.  As a consequence, if $v\in V(\cP_i)\cap V(\cP_j)$ for some $i,j\in [n]$ with $i\neq j$, and $x_v\in \mathrm{supp}(\lt(f))$, for some inner 2-minor $f$ related to an inner interval of $\cP_i$, then $x_v\notin \mathrm{supp}(\lt(g))$, for all inner 2-minors $g$ related to the inner intervals of $\cP_j$.
\noindent  Let $A_\cP=\{v\in V(\cP)\mid v\in V(\cP_i)\cap V(\cP_j), \text{ for some }i,j\in [n], i\neq j\}$. Let $C_\cP$ be the set of vertices $v\in A_\cP$ such that $x_v\notin \mathrm{supp}(\lt(f))$, for all inner 2-minor $f$ related to an inner interval of $\cP$. For each $i\in [n]$, define $D_{\cP_i}$ the set of vertices $v\in V(\cP_i)\cap A_\cP$ such that $ x_v\notin \mathrm{supp}(\lt(f)),\text{ for all inner 2-minor }f \text{ related to an inner interval of }\cP_i$. Denote $d_i=|D_{\cP_i}|$. Set, for all $i\in [n]$, $S'_{\cP_i}=K[x_a : a\in V(\cP_i)\setminus D_{\cP_i} ]$. Then we have:

$$S_\cP/\lt(I_\cP) = K[x_v\mid v\in C_\cP] \otimes_K\left(\bigotimes_{i=1}^{n}S'_{\cP_i}/\lt(I_{\cP_i})\right).$$

\noindent Observe that for all $i\in [n]$ we have $S_{\cP_i}/\mathrm{in}(I_{\cP_i})=S'_{\cP_i}/\mathrm{in}(I_{\cP_i}) \otimes_K K[x_{v}\mid v\in D_{\cP_i}]$. In particular, by \cite[Lemma 5.1.11]{Villareal}, $\rHP_{S'_{\cP_i}/\lt(I_{\cP_i})}(t)= (1-t)^{d_i}\rHP_{K[\cP_i]}(t)$, and since $\cP_i$ is a simple polyomino, by Proposition~\ref{prop:simple_collection-cells-areCM} we have $\rHP_{S'_{\cP_i}/\lt(I_{\cP_i})}(t)=\frac{h_{\cP_i}(t)}{(1-t)^{|V(\cP_i)|-|\cP_i|-d_i}}$. Note that $\sum_{i=1}^n |V(\cP_i)|=|V(\cP)|+|A_\cP|$ and $\sum_{i=1}^n |\cP_i|=|\cP|$. Therefore, by \cite[Lemma 5.1.11]{Villareal}, we have:

 \[
 \rHP_{K[\cP]}(t)=\frac{\prod_{i=1}^{n}h_{K[\cP_{i}]}}{(1-t)^{|V(\cP)|-|\cP|+|A_{\cP}|+|C_{\cP}|-\sum_{i=1}^n d_i}}
 \]

\noindent Observe that $|A_{\cP}|=\sum_{i=1}^n d_i-|C_{\cP}|$. In fact, let $v\in A_{\cP}$. Then there exist $j_1,j_2\in [n]$ such that $v\in V(\cP_{j_1})\cap V(\cP_{j_2})$. Furthermore, if $v\in C_{\cP}$, then $v\in D_{\cP_{j_1}}\cap D_{\cP_{j_2}}$, and if $v\notin C_{\cP}$, then either $v\in D_{\cP_{j_1}}$ or $v\in D_{\cP_{j_2}}$. By this observation, we easily argue the previous equality. So, $\rHP_{K[\cP]}(t)=\frac{\prod_{i=1}^{n}h_{K[\cP_{i}]}}{(1-t)^{|V(\cP)|-|\cP|}}$. By Proposition~\ref{prop:simple_collection-cells-areCM} we have $\dim K[\cP]=|V(\cP)|-|\cP|$, hence $ \prod_{i=1}^{n}h_{K[\cP_{i}]}(t)$ is the $h$-polynomial of $K[\cP]$.
\end{proof}

\noindent By the previous result, we argue that, if $\cP$ is a simple collection of cells, the study of the $h$-polynomial of $K[\cP]$ in terms of the switching rook polynomial of $\cP$ can be reduced to investigate the Hilbert-Poincar\'e series of its connected components. In particular, if for all simple polyominoes \cite[Conjecture 3.2]{Parallelogram Hilbert series} is true, then the same holds for all simple collections of cells.

\begin{coro}\label{cor:rook polynomial of simple collection of cells} 
    Let $\cP$ be a simple collection of cells consisting of $n$ connected components $\cP_1,\dots,\cP_n$. If $h_{K[\cP_i]}(t)$ is equal to the switching rook polynomial of $\cP_i$ for all $i\in [n]$, then $h_{K[\cP]}(t)$ is equal to the switching rook polynomial of $\cP$. 
\end{coro}

\begin{proof} It is a consequence of Lemma~\ref{Lemma: prod of rook polynomial} and Theorem~\ref{thm:product-h-poly}.
   \end{proof}

\begin{exa}\rm 
    Let $\cP$ be a simple collection of cells where every connected component $\cP_i$ is a parallelogram polyomino. Then, by \cite[Theorem 3.5]{Parallelogram Hilbert series}, $h_{K[\cP_i]}(t)$ is the switching rook polynomial of $\cP_i$ for all $i\in [n]$, so $h_{K[\cP]}$ is the switching rook polynomial of $\cP$ by Corollary~\ref{cor:rook polynomial of simple collection of cells}. 
\end{exa}

\noindent As a consequence, we generalize \cite[Theorem 3.13 and Theorem 4.2]{Trento3} to simple thin collections of cells (see Figure \ref{leaf1} and \ref{leaf2} for an example). This is useful in the next sections. 

\noindent We recall the definition of \textit{S-property} given in \cite{Trento3}. Let $\cP$ be a thin collection of cells. A cell $C$ is called \textit{single} if there exists a unique maximal interval of $\cP$ containing $C$. $\cP$ has the \textit{S-property} if every maximal interval of $\cP$ has only one single cell.

\begin{coro}\label{cor:hilbert-series-thin-collections}
Let $\cP$ be a simple thin collection of cells. Then $h_{K[\cP]}$ is the rook polynomial of $\cP$, in particular $\operatorname{reg}(K[\cP])=r(\cP)$. Moreover, $K[\cP]$ is Gorenstein if and only if $\cP$ satisfies the $S$-property. 
\end{coro}
\begin{proof}
    As a consequence of Corollary~\ref{cor:rook polynomial of simple collection of cells} and \cite[Theorem 3.13]{Trento3}, we get that $h_{K[\cP]}$ is the rook polynomial of $\cP$. Moreover, $\operatorname{reg}(K[\cP])=r(\cP)$ since $K[\cP]$ is a Cohen-Macaulay ring. Now we prove the result about Gorensteinness. Let $h_{K[\cP]}(t)=r_{\cP}(t)=\sum_{k=0}^{s}r_kt^k$, where $s=r(\cP)$. Since $K[\cP]$ is a Cohen-Macaulay domain, it is known from \cite{Stanley} that $K[\cP]$ is Gorenstein if and only if $r_i=r_{s-i}$ for all $i=0,\dots,s$, that is, $h_{K[\cP]}(t)$ is palindromic. Assume that $\cP$ consists of $n$ connected components and denote by $h_{K[\cP_i]}(t)$ the $h$-polynomial of $K[\cP_i]$ for all $i\in [n]$. Suppose that $\cP$ has the $S$-property and we prove that $K[\cP]$ is Gorenstein. Since $\cP$ has the $S$-property, then each connected component $\cP_i$ has the $S$-property as well. By \cite[Theorem 4.2]{Trento3}, $h_{K[\cP_i]}(t)$ is palindromic for all $i\in [n]$, so $h_{K[\cP]}(t)$ is palindromic as well, since $h_{K[\cP]}(t)=\prod_{i=1}^{n}h_{K[\cP_{i}]}(t)$ and the product of palindromic polynomials is palindromic. Hence it follows by \cite{Stanley} that $K[\cP]$ is Gorenstein. Now, suppose that $K[\cP]$ is Gorenstein and we show that $\cP$ has the $S$-property. In particular the rook polynomial of $\cP$, that is $h_{K[\cP]}$, is palindromic and observe also that $r_0=1$ and $r_1=|\cP|$.  We assume by contradiction that $\cP$ does not have the $S$-property. For all $i\in [n]$, let $h_{K[\cP_i]}(t)=\sum_{j=0}^{s_i}r_j^{(i)}t^j$ be the $h$-polynomial of $K[\cP_i]$, where $s_i=r(\cP_i)$ and observe that $r_1^{(i)}=|\cP_i|$. Looking at the proof of case (b)$\Rightarrow$(c) of \cite[Theorem 4.2]{Trento3}, for all $k\in [n]$ such that $\cP_k$ has not the $S$-property we have $r_{s_k}>1$ or $r_{s_k-1}>\vert \cP_k\vert$. Therefore, it easily follows that $r_s>1$ or $r_{s-1}>\sum_{i=1}^n\vert \cP_i\vert=\vert \cP\vert$, that is a contradiction, since $h_{K[\cP]}(t)$ is palindromic. 
\end{proof}

\section{Zig-zag collections}

\noindent In this section we discuss the algebraic properties of the coordinate ring of a new family of collections of cells called \textit{zig-zag collections}. Roughly speaking, a zig-zag collection is a collection of cells with just one hole that comes from the intervals related to a zig zag-walk. We start by pointing out that definition.

\begin{defn}\rm\label{Defn: zig-zag collection}
Let $I_1,\dots,I_\ell$ be a sequence of distinct intervals in $\mathbb{Z}^2$. For all $i\in \{1,\ldots,\ell\}$ denote by $\cP_i$ the collection of cells related to $I_i$ and $\cP=\bigcup_{i=1}^{\ell} \cP_i$. $\cP$ is a \emph{zig-zag collection} if satisfies:
\begin{enumerate}
	\item $I_1\cap I_\ell=\{v_1=v_{\ell+1}\}$ and $I_i\cap I_{i+1}=\{v_{i+1}\}$, for all $i=1,\dots,\ell-1$;
	\item $v_i$ and $v_{i+1}$ are on the same edge interval of $\cP$, for all $i=1,\dots,\ell$.
    \item $I_{i}\cap I_{j}=\emptyset$ for all $\{i,j\}\subseteq [\ell]$ with $i<j$, such that $j\neq i+1$ and $(i,j)\neq (1,\ell)$.
\end{enumerate}
In such a case, we say that $\cP$ is supported by $I_1,\dots,I_\ell$.  
\end{defn}

\noindent In Figure \ref{Figure: example zig_zag coll} (A) a zig-zag collection is shown. In Figure \ref{Figure: example zig_zag coll} (B) we figure out an example of a non zig-zag collection, in fact $\vert I_1\cap I_4\vert =4.$
  \begin{figure}[h!]
    \centering	
    \subfloat[A zig-zag collection.]{\includegraphics[scale=0.65]{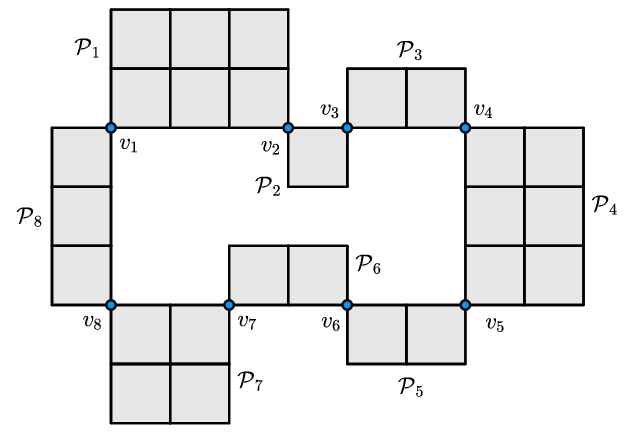}}\qquad
    \subfloat[A non zig-zag collection.]{\includegraphics[scale=0.65]{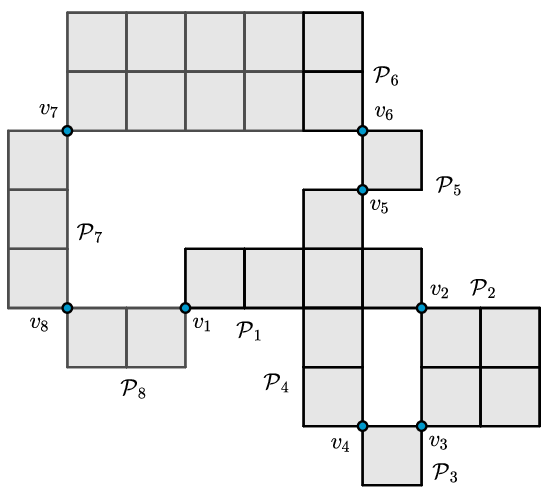}}
    \caption{Collections of cells containing a zig-zag walk.}
    \label{Figure: example zig_zag coll}
    \end{figure} 

\begin{rmk}\rm
Let $\cP$ be a zig-zag collection. From Definition \ref{Defn: zig-zag collection} it follows that $\cP$ is a non-simple weakly connected collection of cells having just one hole and $\ell$ is an even number.\\
Moreover, $I_{\cP}$ is not prime because $I_1,\dots,I_{\ell}$ is a zig-zag walk of $\cP$. For more details we refer to \cite[Section 3]{Trento}, where the arguments given for polyominoes can be extended easily to collections of cells. 
\end{rmk}

\noindent Our first step is to study the Gr\"obner basis of the inner 2-minor ideal of a zig-zag collection. For this aim, we introduce four total orders on $\Z^2$. Let $(i,j),(k,l)\in \Z^2$, then we say that: 
\begin{itemize}

    \item $(k,l)<^{(1)}(i,j)$, if $l<j$, or $l=j$ and $i<k$;
    \item $(k,l)<^{(2)}(i,j)$, if $l<j$, or $l=j$ and $k<i$;
    \item $(k,l)<^{(3)}(i,j)$, if $l>j$, or $l=j$ and $k<i$;
    \item $(k,l)<^{(4)}(i,j)$, if $l>j$, or $l=j$ and $i<k$.
    
\end{itemize}

\noindent Let $\cP$ be a collection of cells. For $i\in \{1,2,3,4\}$, denote by $<_{\mathrm{lex}}^{(i)}$ the lexicographic order on $K[x_v\mid v\in V(\cP)]$, induced by the following total order on the variables: $x_a <_{\mathrm{lex}}^{(i)} x_b$ if and only if $a <^{(i)} b$ for $a,b \in V(\cP)$. If $f$ is an inner 2-minor in $I_\cP$, observe that for $i\in \{1,3\}$ the initial monomial of $f$ with respect to $<_{\mathrm{lex}}^{(i)}$ is related to the anti-diagonal corners of the associated inner interval. While for $i\in \{2,4\}$ the initial monomial of $f$ is related to the diagonal corners. Moreover, \cite[Theorem 4.1]{Qureshi} and the arguments of its proof hold also considering the monomial orders $<_{\mathrm{lex}}^{(i)}$ for $i\in\{2,4\}$. The same can be considered for \cite[Remark 4.2]{Qureshi} with respect to the monomial orders $<_{\mathrm{lex}}^{(i)}$ for $i\in \{1,3\}$. Similar considerations can be made using the reverse lexicographic order, with reference to \cite[Proposition 3.2]{Trento2}. For our aim, we state the result below, which can be proved by the same arguments of the mentioned results in \cite{Qureshi} or also using the lemmas in \cite[Section 3]{Cisto_Navarra_CM_closed_path}.

\begin{prop}\label{prop:s-polynomial}
Let $\cP$ be a collection of cells and $f,g\in I_\cP$ inner 2-minors. Assume $f$ is related to the inner interval $[a,b]$, with anti-diagonal corners $c,d$ and $g$ is related to the inner interval $[\alpha,\beta]$, with anti-diagonal corners $\gamma,\delta$. Without loss of generality, assume the second coordinate of $a$ is smaller or equal to the second coordinate of $\alpha$, $c$ belongs to the same vertical edge interval of $a$ and $\gamma$ belongs to the same horizontal edge interval of $\alpha$. For $i\in\{1,\ldots,4\}$ consider $S^{(i)}(f,g)$ be the S-polynomial of $f,g$ with respect to $<^{(i)}_{\mathrm{lex}}$. Then the following holds:  
\begin{itemize}
\item For $i\in \{2,4\}$, $S^{(i)}(f,g)$ does not reduce to 0 modulo $G(\cP)$ with respect to $<^{(i)}_{\mathrm{lex}}$ if and only $b=\alpha$ and $[a,\gamma], [a,\delta]$ are not inner intervals of $\cP$ (see Figure~\ref{fig:diagonal}).
\item For $i\in \{1,3\}$, $S^{(i)}(f,g)$ does not reduce to 0 modulo $G(\cP)$ with respect to $<^{(i)}_{\mathrm{lex}}$ if and only $c=\gamma$ and the sets $\{d,\alpha\}, \{d,\beta\}$ are not the set of anti-diagonal corners of inner intervals of $\cP$ (see Figure~\ref{fig:anti-diagonal}). 
\end{itemize}
\begin{figure}[h]
        \centering
        \subfloat[]{\includegraphics[scale=0.5]{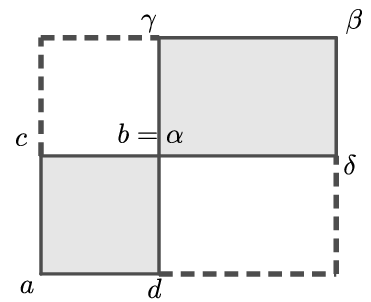}\label{fig:diagonal}}\qquad \qquad
        \subfloat[]{\includegraphics[scale=0.5]{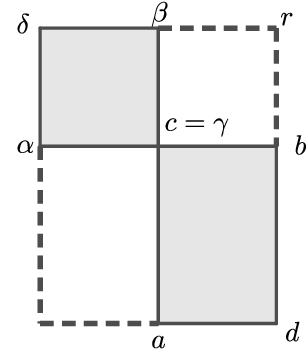}\label{fig:anti-diagonal}}
        \caption{Configurations related to conditions of Proposition~\ref{prop:s-polynomial}.}
        \label{Figure:S-polynomial}
    \end{figure}
\end{prop}
 
\begin{lemma}
Let $\cP$ be a zig-zag collection supported by $I_1,\dots,I_\ell$. Then there exists a monomial order $\prec$ on $S_{\cP}$ such that:
\begin{enumerate}
    \item the set of generators $G(\cP)$ of $I_\cP$ forms the reduced Gr\"obner basis of $I_{\cP}$ with respect $\prec $;
    \item the order $\prec $ induces a monomial order $<_i$ on $S_{\cP_i}$ for all $i\in[\ell]$, such that for all $k\in [\ell/2]$ we have:
\begin{itemize}
    \item $\mathrm{in}_{<_{2k-1}}(I_{\cP_{2k-1}})=(\{x_a x_b\mid a,b\ \mbox{are anti-diagonal corners of an inner interval of}\ I_{2k-1}\})$;
    \item $\mathrm{in}_{<_{2k}}(I_{\cP_{2k}})=(\{x_c x_d\mid c,d\ \mbox{are diagonal corners of an inner interval of}\ I_{2k}\})$;
    \item $\mathrm{in}_\prec (I_\cP)=\sum_{i=1}^{\ell}\mathrm{in}_{<_i}(I_{\cP_i})$. 
\end{itemize}
\end{enumerate}
In particular, $I_\cP$ is radical.

\label{grobner}
\end{lemma}

\begin{proof}
    
    \noindent  We start by defining inductively a suitable total order $<_{V(\cP)}$ on $V(\cP)$. It is not restrictive to suppose that $I_1$, $I_2$ and $I_\ell$ are arranged as in Figure \ref{Figure: arrangement I1, I2 Il-a}, otherwise it is sufficient to do some suitable reflections of $\cP$ or to relabel the intervals $\{I_i\}_{i\in [\ell]}$.  
    
    \begin{figure}[h]
        \centering
        \subfloat[]{\includegraphics[scale=0.7]{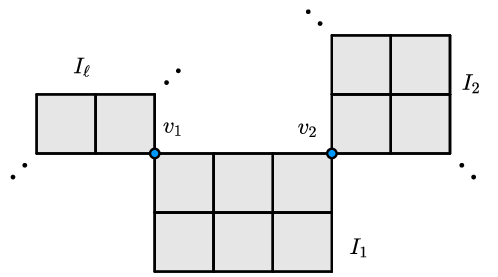}\label{Figure: arrangement I1, I2 Il-a}}\qquad
        \subfloat[]{\includegraphics[scale=0.7]{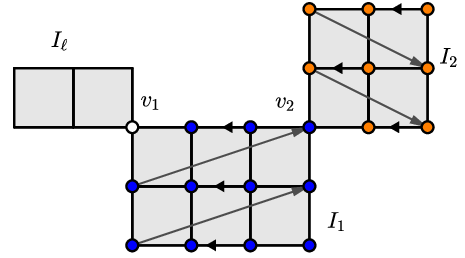}\label{Figure: arrangement I1, I2 Il-b}}
        \caption{Arrangement of $I_1$, $I_2$ and $I_\ell$ and the related orders.}
        \label{Figure: arrangement I1, I2 Il}
    \end{figure}

    \noindent In the first step we give a total order on $(I_{1}\setminus \{v_1\}) \cup I_{2}$. We use $<^{(3)}$ and $<^{(2)}$, respectively, to order the vertices in $I_{1}\setminus \{v_1\}$ and $I_{2}\setminus \{v_2\}$; look at Figure \ref{Figure: arrangement I1, I2 Il-b}, where the arrows denote the sense of the orders of the vertices in $I_{1}\setminus \{v_1\}$ and $I_{2}\setminus \{v_2\}$ respectively. Moreover, we set that every vertex of $\cP$ in $I_{1}\setminus \{v_1\}$ is smaller than each one in $I_{2}\setminus \{v_2\}$. Let $k\in[\ell/2]$ with $k\geq 2$ and consider the intervals $I_{2k-2}$, $I_{2k-1}$ and $I_{2k}$. Assume that a total order in $I_{2k-2}\setminus \{v_{2k-2}\}$ is already defined; we want to define a total order on $(I_{2k-1}\setminus \{v_{2k-1}\})\cup I_{2k}$. First of all, we introduce two total orders on $I_{2k-1}\setminus \{v_{2k-1}\}$ and on $I_{2k}\setminus \{v_{2k}\}$ in the following way.
    \begin{enumerate}
           \item If $I_{2k-2}$, $I_{2k-1}$ and $I_{2k}$ are arranged as in Figures \ref{Figure: horizontal arrangement I_{2k-2}, I_{2k-1} and I_{2k}-a} and \ref{Figure: horizontal arrangement I_{2k-2}, I_{2k-1} and I_{2k}-d} or as in Figures \ref{Figure: vertical arrangement I_{2k-2}, I_{2k-1} and I_{2k}-a}, \ref{Figure: vertical arrangement I_{2k-2}, I_{2k-1} and I_{2k}-b} and \ref{Figure: vertical arrangement I_{2k-2}, I_{2k-1} and I_{2k}-d}, then we use $<^{(1)}$ and $<^{(2)}$ to order the vertices in $I_{2k-1}\setminus \{v_{2k-1}\}$ and $I_{2k}\setminus \{v_{2k}\}$ respectively.
           
           \item In the case that $I_{2k-2}$, $I_{2k-1}$ and $I_{2k}$ are as in Figure \ref{Figure: horizontal arrangement I_{2k-2}, I_{2k-1} and I_{2k}-b}, then we define an order of the vertices in $I_{2k-1}\setminus \{v_{2k-1}\}$ and $I_{2k}\setminus \{v_{2k}\}$ using $<^{(3)}$ and $<^{(2)}$, respectively.
           \item When $I_{2k-2}$, $I_{2k-1}$ and $I_{2k}$ are in the position of Figure \ref{Figure: horizontal arrangement I_{2k-2}, I_{2k-1} and I_{2k}-c}, then use respectively $<^{(1)}$ and $<^{(4)}$ to make an order in $I_{2k-1}\setminus \{v_{2k-1}\}$ and $I_{2k}\setminus \{v_{2k}\}$.
           
           \item If $I_{2k-2}$, $I_{2k-1}$ and $I_{2k}$ are in the position of Figure~\ref{Figure: vertical arrangement I_{2k-2}, I_{2k-1} and I_{2k}-c} then apply respectively $<^{(3)}$ and $<^{(4)}$ to define an order in $I_{2k-1}\setminus \{v_{2k-1}\}$ and $I_{2k}\setminus \{v_{2k}\}$.
    \end{enumerate}
     
     \begin{figure}[h]
        \centering
        \subfloat[]{\includegraphics[scale=0.7]{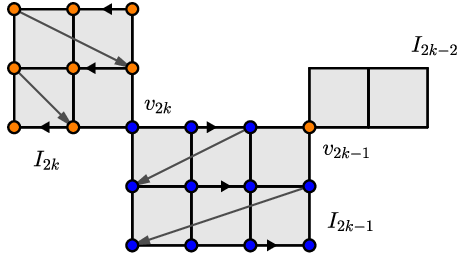}\label{Figure: horizontal arrangement I_{2k-2}, I_{2k-1} and I_{2k}-a}}\qquad
        \subfloat[]{\includegraphics[scale=0.7]{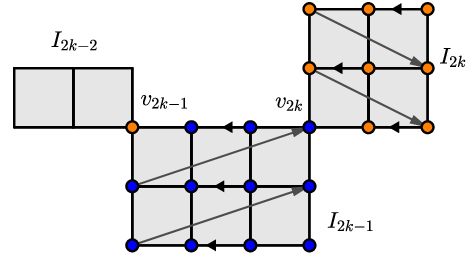}\label{Figure: horizontal arrangement I_{2k-2}, I_{2k-1} and I_{2k}-b}}\qquad\qquad
        \subfloat[]{\includegraphics[scale=0.7]{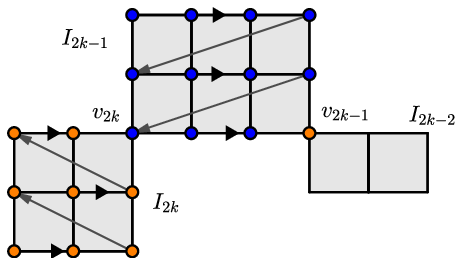}\label{Figure: horizontal arrangement I_{2k-2}, I_{2k-1} and I_{2k}-c}}\qquad
        \subfloat[]{\includegraphics[scale=0.7]{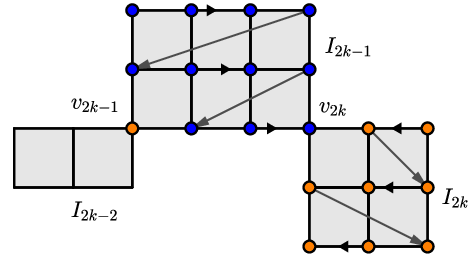}\label{Figure: horizontal arrangement I_{2k-2}, I_{2k-1} and I_{2k}-d}}
        \caption{Horizontal arrangements of $I_{2k-2}$, $I_{2k-1}$ and $I_{2k}$ and the related orders.}
        \label{Figure: horizontal arrangement I_{2k-2}, I_{2k-1} and I_{2k}}
    \end{figure}

    \begin{figure}[h]
        \centering
        \subfloat[]{\includegraphics[scale=0.7]{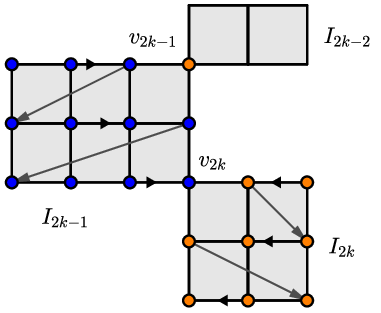}\label{Figure: vertical arrangement I_{2k-2}, I_{2k-1} and I_{2k}-a}}\qquad\qquad
        \subfloat[]{\includegraphics[scale=0.7]{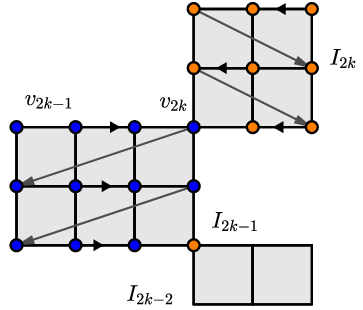}\label{Figure: vertical arrangement I_{2k-2}, I_{2k-1} and I_{2k}-b}}\qquad\qquad\qquad\qquad
        \subfloat[]{\includegraphics[scale=0.7]{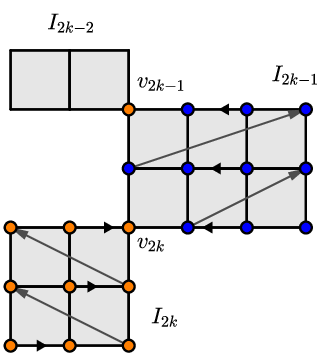}\label{Figure: vertical arrangement I_{2k-2}, I_{2k-1} and I_{2k}-c}}\qquad\qquad
        \subfloat[]{\includegraphics[scale=0.7]{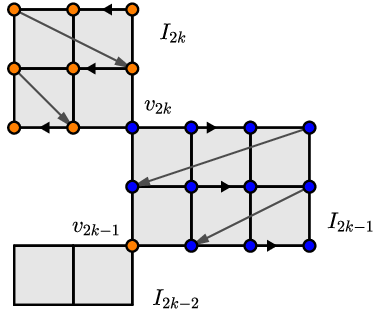}\label{Figure: vertical arrangement I_{2k-2}, I_{2k-1} and I_{2k}-d}}
        \caption{Vertical arrangements of $I_{2k-2}$, $I_{2k-1}$ and $I_{2k}$ and the related orders.}
        \label{Figure: vertical arrangement I_{2k-2}, I_{2k-1} and I_{2k}}
    \end{figure}

    \noindent Therefore, depending on the described situations, we have two total orders on $I_{2k-1}\setminus \{v_{2k-1}\}$ and on $I_{2k}\setminus \{v_{2k}\}$; moreover, putting that every vertex of $\cP$ in $I_{2k-1}\setminus \{v_{2k-1}\}$ is smaller than everyone in $I_{2k}\setminus \{v_{2k}\}$, we get a total order on $(I_{2k-1}\setminus \{v_{2k-1}\})\cup I_{2k}$.
    Applying inductively the procedure until $k=\ell/2$, we can get a total order $<_{V(\cP)}$ for the set of the vertices of $\cP$. 
    Now, define the lexicographic order $<_{\mathrm{lex}}$ on $S_{\cP}$ induced by the order on the variables: $x_b<_{\mathrm{lex}} x_a$, where $a,b\in V(\cP)$, if $b<_{V(\cP)} a$. We prove that $<_{\mathrm{lex}}$ provides the desired claim.\\ 
    We start by proving (1). Let $f=x_ax_b-x_cx_d$ and $g=x_px_q-x_rx_s$ be two generators of $I_{\cP}$ attached to the inner intervals $[a,b]$ and $[p,q]$ of $\cP$, respectively. We want to show that the $S$-polynomial $S(f,g)$ of $f$ and $g$ reduces to $0$ with respect to $G(\cP)$. Let $k\in[\ell/2]$. We may consider just the case when $I_{2k-2}$, $I_{2k-1}$ and $I_{2k}$ are arranged as in Figures \ref{Figure: horizontal arrangement I_{2k-2}, I_{2k-1} and I_{2k}-b} (with the convention that $I_0=I_\ell$); indeed, all the other situations described in Figures \ref{Figure: horizontal arrangement I_{2k-2}, I_{2k-1} and I_{2k}} and \ref{Figure: vertical arrangement I_{2k-2}, I_{2k-1} and I_{2k}} can be proved similarly. If $[a,b],[p,q]\subseteq I_{2k-1}$ or $[a,b],[p,q]\subseteq I_{2k}$ then $S(f,g)$ reduces to $0$ with respect to $G(\cP)$ by Proposition~\ref{prop:s-polynomial}. Assume that $[a,b]\subseteq I_{2k-1}$ and $[p,q]\subseteq I_{2k}$. 
    In such a case, observe that $[a,b]\cap [p,q]=\emptyset$ or $[a,b]\cap [p,q]=\{b\}$ with $b=p$. Since in this case we have also $\operatorname{in}_<(f)=x_cx_d$ and $\operatorname{in}_<(g)=x_px_q$, then in both cases we obtain $\gcd(\operatorname{in}_<(f),\operatorname{in}_<(g))=1$. Hence $S(f,g)$ reduces to $0$ with respect to $G(\cP)$. It is trivial to see that $\gcd(\operatorname{in}_<(f),\operatorname{in}_<(g))=1$ occurs when $[a,b]\subseteq I_i$ and $[p,q]\subseteq I_j$ with $i,j \in[\ell]$ and $j>i+1.$ In conclusion, the claim (1) is completely proved.\\ 
    Now, observe that $<_{\mathrm{lex}}$ induces in a natural way a monomial order $<_i$ on $S_{\cP_i}$, for $i\in[\ell]$, which is the restriction of $<_{\mathrm{lex}}$ on $S_{\cP_i}$. In particular, the claim (2) follows from the considerations done for claim (1).
    
    \noindent Finally, it is known that if the initial ideal of a homogeneous ideal $I$ of $K[x_1,\dots,x_n]$, with respect to a monomial order, is squarefree then $I$ is radical (see \cite[Corollary 2.2]{2.n}). Since $\lt_{<_\mathrm{lex}}(I_{\cP})$ is squarefree, then $I_\cP$ is radical.
\end{proof}

\begin{rmk}\label{rem:order-zig-zag} \rm
Let $\cP$ be a zig-zag collection and $\prec$ be the monomial order satisfying the conditions of Lemma~\ref{grobner}. Denote $\mathcal{E}_i=I_i\setminus \{v_i\}$ for all $i\in [\ell]$. We highlight that, by the proof of Lemma \ref{grobner}, the order $\prec$ is the lexicographic order induced by a total order on $V(\cP)$ and it satisfies the following property: if $a\in \mathcal{E}_i$ and $b\in \mathcal{E}_j$ with $i<j$ then $x_a \prec x_b$. This observation will be useful in later arguments.
\end{rmk}

For shortness, we will write $\mathrm{in}(I_\cP)$ and $\mathrm{in}(I_{\cP_i})$ for $\mathrm{in}_{\prec}(I_\cP)$ and $\mathrm{in}_{<_i}(I_{\cP_i})$ respectively. Moreover, following the construction in Lemma~\ref{grobner}, without loss of generality we can always assume that $v_1$ is an anti-diagonal corner in $I_1$. For all $i\in \{1\ldots,\ell\}$ denote by $S'_i=K[x_a \mid a\in V(I_i)\setminus \{v_{i}\}]$. Set also $\rHP_{K[\cP_i]}(t)=\frac{h_i(t)}{(1-t)^{n_i}}$, in particular we know that $n_i=|V(\cP_i)|-|\cP_i|$. Finally, before proving the main result of this section, we recall a beautiful result achieved in \cite{conca4}, which will be crucial also in the other parts of this work.

\begin{prop}[\cite{conca4}]
Let $I$ be a homogeneous ideal in a polynomial ring $S$ and suppose that $\mathrm{in}(I)$ is radical with respect to some term order. Then $\mathrm{reg}(S/I)=\mathrm{reg}(S/\mathrm{in}(I))$ and $\mathrm{depth}(S/I)=\mathrm{depth}(S/\mathrm{in}(I))$. Furthermore, $S/I$ is Cohen-Macaulay if and only if $S/\mathrm{in}(I)$ is Cohen-Macaulay.
\label{conca}
\end{prop}

\begin{thm}
Let $\cP$ be a zig-zag collection supported by $I_1,\dots,I_\ell$. Then:
\begin{enumerate}
    \item $$\rHP_{K[\cP]}(t)=(1-t)^\ell\prod_{i=1}^\ell \rHP_{K[\cP_i]}(t)=\frac{h_1(t)\cdot h_2(t) \cdots h_\ell(t)}{(1-t)^{|V(\cP)|-|\cP|}},$$
where $h_i(t)$ is the $h$-polynomial of $K[\cP_i].$
\item  $K[\cP]$ is Cohen-Macaulay with Krull dimension $|V(\cP)|-|\cP|$.
\end{enumerate}
\label{zig-collection}
\end{thm}
\begin{proof}
(1) By Lemma~\ref{grobner} we obtain that $S_\cP/\mathrm{in}(I_\cP)=\otimes_{i=1}^{\ell} S'_{\cP_i}/\mathrm{in}(I_{\cP_i})$. Observe that for all $i\in [\ell]$ then $S_{\cP_i}/\mathrm{in}(I_{\cP_i})=S'_{\cP_i}/\mathrm{in}(I_{\cP_i}) \otimes_K K[x_{v_{i}}]$. So we argue that $\rHP_{S'_{\cP_i}/\lt(I_{\cP_i})}(t)= (1-t)\rHP_{K[\cP_i]}(t)$ and we obtain our claim on $\rHP_{K[\cP]}(t)$. \\
(2) By \cite[Theorem 2.2]{Qureshi} we know that $K[\cP_i]$ is a Cohen-Macaulay domain of dimension $|V(\cP_i)|-|\cP_i|$, so by Proposition~\ref{conca} also $S_{\cP_i}/\mathrm{in}(I_{\cP_i})$ is Cohen-Macaulay of dimension $|V(\cP_i)|-|\cP_i|$. Therefore, by \cite[Lemma 3.1.35]{Villareal} we obtain that $S'_{\cP_i}/\mathrm{in}(I_{\cP_i})$ is Cohen-Macaulay of dimension $|V(\cP_i)|-|\cP_i|-1$ and, as a consequence, $S_{\cP}/\mathrm{in}(I_{\cP})$ is Cohen-Macaulay of dimension $\sum_{i=1}^{\ell} (|V(\cP_i)|-|\cP_i|-1)=|V(\cP)|-|\cP|$. By Proposition~\ref{conca} the same property holds for $K[\cP]$.
\end{proof}

\begin{coro}\label{Coro: For zig-zag coll, h is the switching rook pol}
Let $\cP$ be a zig-zag collection supported by $I_1,\dots,I_\ell$. Then $$\rHP_{K[\cP]}(t)=\frac{h(t)}{(1-t)^{|V(\cP)|-|\cP|}},$$ where $h(t)$ is the switching rook polynomial of $\cP$. Moreover $\mathrm{reg}(K[\cP])=r(\cP)$.
\end{coro}

\begin{proof}
    For $i\in[\ell]$ we denote by $h_i(t)$ the $h$-polynomial of $K[\cP_i]$, so from \cite[Theorem 3.5]{Parallelogram Hilbert series} we have that $h_i(t)$ is the switching rook-polynomial of $\cP_i$ for all $i\in [\ell]$. Recall that Theorem \ref{zig-collection} states in particular that $h(t)=\prod_{i=1}^{\ell} h_i(t)$. From Lemma \ref{Lemma: prod of rook polynomial} we have that $\prod_{i=1}^{\ell} h_i(t)$ is the switching rook polynomial of $\cup_{i=1}^\ell \cP_i$, that is $h(t)$ is the switching rook polynomial of $\cP$. Moreover, since $K[\cP]$ is Cohen-Macaulay from Theorem \ref{zig-collection}, then it follows that $\mathrm{reg}(K[\cP])=r(\cP)$. 
\end{proof}

\noindent Moreover, if $I$ is an interval of $\ZZ^2$ and $\cP_I$ is the attached polyomino then we say that $\cP_I$ is a \textit{square} if $I=[a,a+n(1,1)]$ for $a\in \ZZ^2$ and for $n>0.$ If $n=1$, then $I$ is a cell of $\ZZ^2$.  

\begin{coro}\label{cor:gor-zig}
    Let $\cP$ be a zig-zag collection supported by $I_1,\dots,I_\ell$. 
    \begin{enumerate}
        \item If $K[\cP]$ is Gorenstein, then $\cP_i$ is a square for all $i\in [\ell]$.
        \item If $\cP_i$ is a cell for all $i\in [\ell]$, then $K[\cP]$ is Gorenstein.
    \end{enumerate}
 
    \end{coro}
\begin{proof}
    (1) Let $h(t)=\sum_{k=1}^{s} h_kt^k$ be the $h$-polynomial of $K[\cP]$. From Corollary \ref{Coro: For zig-zag coll, h is the switching rook pol} we have that $h(t)$ is the switching rook polynomial of $\cP$ and $s=r(\cP)$. Since $K[\cP]$ is Gorenstein, it follows from \cite[Corollary 5.3.10]{Villareal} that $h_k=h_{s-k}$ for all $k\in [s]$. In particular, we have that $h_s=1$. Suppose by contradiction that there exists $h\in [\ell]$ such that $\cP_h$ is not a square. Up to reflections, rotations or translations of $\cP$, we may assume that $\cP_h=\cP_{[(1,1),(m.n)]}$, where $m>n$. Observe that $r(\cP)=n-1$. We define a canonical configuration $\cT_1$ of $n-1$ rooks in $\cP_h$, placing a rook in the cell of $\cP_h$ having lower left corner $(j,j)$ for all $j\in [n-1]$. Moreover, since $m>n$, we can define another canonical configuration $\cT_2$ of $n-1$ rooks in $\cP_h$, placing a rook in the cell of $\cP_h$ having lower left corner $(j,j)$ for all $j\in [n-2]$ and a rook in that one with $(n,n-1)$ as lower left corner. Denote by $\cC_i$ a canonical configuration of $r(\cP_i)$ rooks in $\cP_i$, for all $i\in [s]\setminus\{h\}$, and set $\cC=\cup_{i\in[s]\setminus\{h\}} \cC_i$. It is easy to see that $\cC\cup \cT_1$ and $\cC\cup\cT_2$ are two canonical configurations of $r(\cP)$ rooks in $\cP$, so $h_s>1$. This is a contradiction because $h_s=1$. In conclusion $\cP_i$ is a square for all $i\in [\ell]$.\\
    (2) Observe that $S_\cP/\mathrm{in}(I_\cP)$ can be viewed as an edge ring of a graph $G$ on the vertex set $V(\cP)$, where $\{a,b\}\in E(G)$ if either $a,b$ are the anti-diagonal corners of an inner interval of $I_{2k}$ or they are the diagonal corners of an inner interval of $I_{2k-1}$, for $k\in[\ell/2]$. If $\cP_i$ is a cell for all $i\in [\ell]$,  $G$ is a disjoint union of edges, so $S_\cP/\mathrm{in}(I_\cP)$ is Gorenstein by \cite[Corollary 9.3.3]{H_H_monomial_ideals}. Therefore, it follows that $K[\cP]$ is Gorenstein by \cite[Proposition 9.6.17]{Villareal}.
\end{proof}

\begin{rmk}\label{rmk:questio_gor-zig}\rm 
    If there exists $j\in [\ell]$ such that $\cP_j$ is a square different from a cell, then $S_\cP/\mathrm{in}(I_\cP)$ cannot be Gorenstein because $G$ is a chordal graph but it is not a disjoint union of edges. Anyway, we tested different examples with \texttt{Macaulay2} (\cite{Package_M2,M2}) that shows $K[\cP]$ is still Gorenstein. So, we ask if the converse of (1) is true in general. 
\end{rmk}

\section{Cohen-Macaulay property and Hilbert-Poincar\'e series of closed paths with zig-zag walks}

\noindent In this section we investigate Cohen-Macaulayness and the Hilbert-Poincar\'e series of non-prime closed path polyominoes. From now onward, we assume that $\cP$ is a non-prime closed path. 
By~\cite[Theorem 6.2]{Cisto_Navarra_closed_path}, $\cP$ contains a zig-zag walk. See Figure \ref{Figure: example closed path with zig-zag Big} for an example of a closed path with zig-zag walks.

\begin{discussion}\label{discussion1} \rm
If $\cP$ is a closed path polyomino having zig-zag walks, by \cite[Proposition 6.1]{Cisto_Navarra_closed_path} $\cP$ does not contain any $L$-configuration or ladders with at least three steps. So, any two maximal inner intervals of $\cP$ intersect themselves in the cells displayed in Figure~\ref{fig:threestepladder} or~\ref{fig:Lshape} (up to reflections or rotations). 
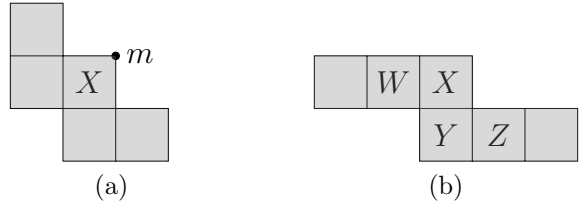
\begin{figure}[h]
	
	\subfloat[]{
		\begin{tikzpicture}[scale=0.5]
			\draw[] (0,2)--(2,2)--(2,0)--(1,0)--(1,1)--(0,1)--(0,2) 
			(1,2)--(1,1)--(2,1)
			(2,0)-- (3,0)--(3,1)--(2,1)
			(0,2)--(0,3)--(1,3)--(1,2);
			\filldraw[black] (2,2) circle (2.0pt) node[anchor=west]  {$m$};
			\filldraw[black] (1.5,1.5) circle (0pt) node[anchor=center]  {$X$};

			\fill[fill=gray, fill opacity=0.3] (0,1)-- (0,3)--(1,3)--(1,2)--(2,2)--(2,1)--(3,1)--(3,0)--(1,0)--(1,1);
		\end{tikzpicture}
	\label{fig:threestepladder}
	}\qquad\qquad
	\subfloat[]{
		\begin{tikzpicture}[scale=0.5]
			\draw[] (0,2)--(2,2)--(2,0)--(1,0)--(1,1)--(0,1)--(0,2) 
			(1,2)--(1,1)--(2,1)
			(2,0)-- (4,0)--(4,1)--(2,1) (3,0)--(3,1)
			(0,2)--(-1,2)--(-1,1)--(0,1);
            \filldraw[black] (.5,1.5) circle (0pt) node[anchor=center]  {$W$};
            \filldraw[black] (1.5,1.5) circle (0pt) node[anchor=center]  {$X$};
            \filldraw[black] (1.5,.5) circle (0pt) node[anchor=center]  {$Y$};
 			\filldraw[black] (2.5,.5) circle (0pt) node[anchor=center]  {$Z$};
			\fill[fill=gray, fill opacity=0.3] (1,0)-- (4,0)--(4,1)--(1,1);
			\fill[fill=gray, fill opacity=0.3] (2,2)-- (-1,2)--(-1,1)--(2,1);
	\end{tikzpicture}
	\label{fig:Lshape}}
	\caption{Possible ``changes of direction'' in a closed path having a zig-zag walk.}
	
\end{figure}

\noindent Furthermore, by \cite{Cisto_Navarra_CM_closed_path} we know that there exists a monomial order such that the Gr\"obner basis is quadratic and squarefree. The monomial order that provides this property can be built following \cite[Algorithm 4.1]{Cisto_Navarra_CM_closed_path}. Following this construction, we obtain that for all inner intervals containing the vertex $m$ as in Figure~\ref{fig:threestepladder}, (up to reflection and rotation), the initial term of the related inner 2-minors contain the variable $x_m$ in its support. From now onward, we denote by $M_\cP$ the set of all vertices $m$ in $\cP$ as in Figure~\ref{fig:threestepladder}, up to rotations and reflections, and by $C_m$ the cell labeled with $X$ with reference to the same figure.
\end{discussion}

\begin{stp}\label{setup1}\rm
Let $\cP$ be a closed path polyomino. With reference to Discussion~\ref{discussion1}, assume $M_\cP=\{m_1,\ldots,m_r\}$ for some $r\in \mathbb{N}$. Then we define for all $i\in[r]$: 
\begin{enumerate}
\item $\cQ_0=\cP$ and $\cQ_i=\cP\setminus \{C_{m_1},\ldots,C_{m_i}\}$;
\item $\cQ'_i=\cQ_{i-1}\setminus\{C\in \cP\mid C\ \mbox{belongs to a maximal cell interval containing}\ C_{m_i}\}.$
\end{enumerate}
\end{stp}

\begin{rmk}\label{rem:initialQi} \rm
    Let $\cP$ be a closed path with a zig-zag walk and $\prec$ be the monomial order on $S_\cP$ defined by \cite[Algorithm 4.1]{Cisto_Navarra_CM_closed_path}, that provides a quadratic Gr\"obner basis for $I_\cP$. For all $i\in [r]$, with reference to Set-up~\ref{setup1}, since $V(\cQ_i)\subset V(\cP)$ and $V(\cQ_i')\subset V(\cP)$, the monomial order $\prec$ can be considered also in $S_{\cQ_i}$ and $S_{\cQ'_{i}}$. In particular, $\prec$ trivially preserves the same initial terms in the ideals $I_\cP$, $I_{\cQ_i}$ and $I_{\cQ'_i}$. Furthermore, by \cite[Proposition 4.1]{Cisto_Navarra_Hilbert_series}, $\prec$ provides a quadratic squarefree Gr\"obner basis also for $I_{\cQ_i}$ and $I_{\cQ_i'}$. So, with abuse of notation, we refer to $\mathrm{in}(I_{\cQ_{i}})$, $\mathrm{in}(I_{\cQ_{i-1}})$ and $\mathrm{in}(I_{\cQ'_{i}})$ considering the monomial order $\prec$.
\end{rmk}

\begin{exa} \rm 
 In Figure \ref{Figure: some Q and Q_i'} we show an example of a closed path $\cP$ and the related collection of cells $\cQ_r$ and $\cQ_2'$, with reference to Set-up~\ref{setup1}.
 \begin{figure}[h!]
    \centering	
    \subfloat[$\cP$]{\includegraphics[scale=0.50]{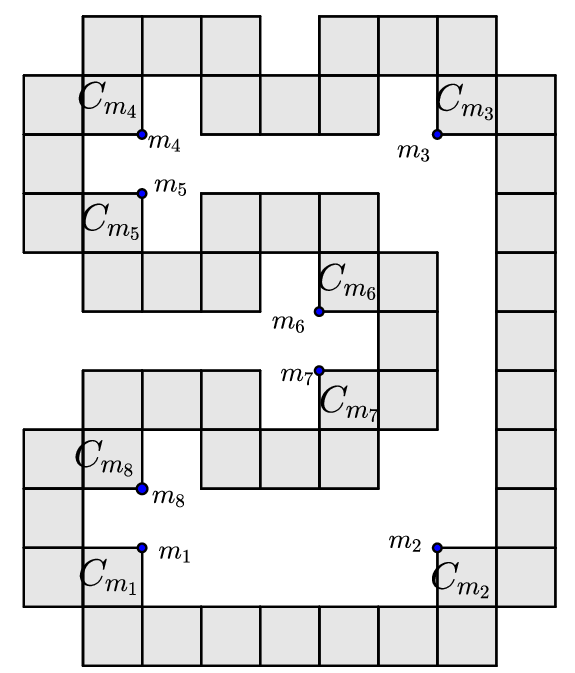}}\qquad
    \subfloat[$\cQ_r$]{\includegraphics[scale=0.50]{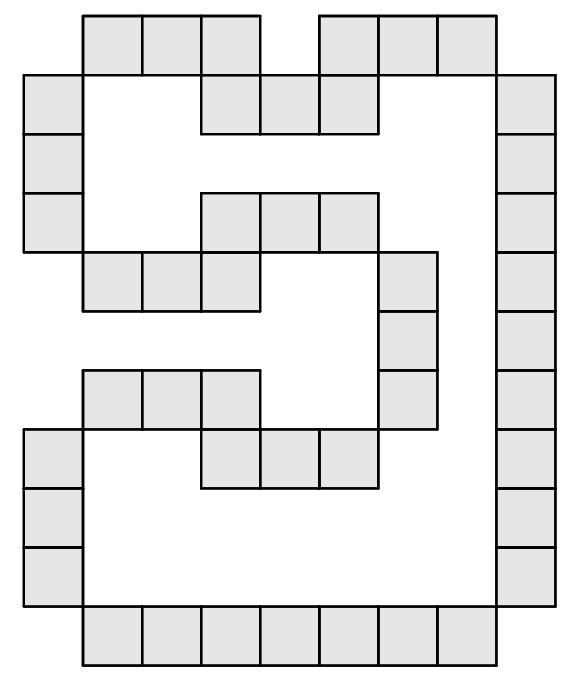}}\qquad
    \subfloat[$\cQ_2'$]{\includegraphics[scale=0.50]{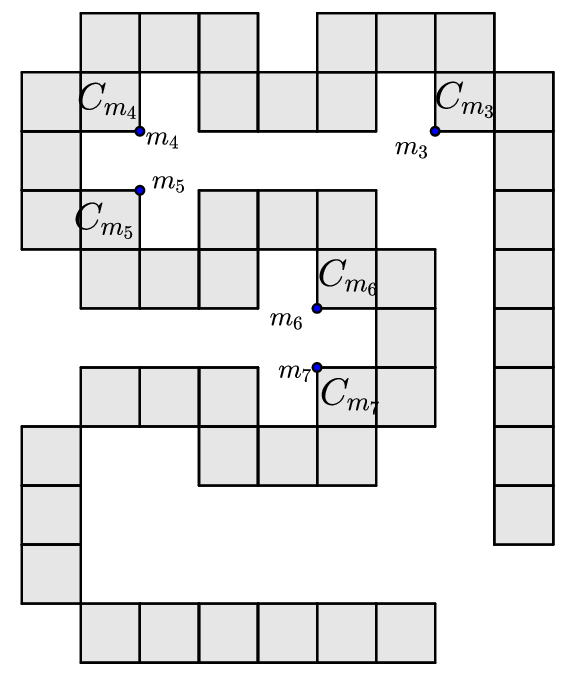}}
    \caption{An example of a polyomino $\cP$ and the related $\cQ_r$, $\cQ_4$ and $\cQ_2'$.}
    \label{Figure: some Q and Q_i'}
    \end{figure} 
\end{exa}

\begin{lemma} 
Let $\cP$ be closed path polyomino. In the framework of Set-up~\ref{setup1} we have:
$$\mathrm{HP}_{K[\cP]}(t)=\mathrm{HP}_{K[\cQ_r]}(t)+\sum_{i=1}^{r}\frac{t}{1-t}\mathrm{HP}_{K[\cQ'_i]}(t)=\mathrm{HP}_{K[\cQ_r]}(t)+\frac{t}{(1-t)^{|V(\cP)|-|\cP|}}\left(\sum_{i=1}^{r}r_{\cQ'_i}(t)\right)$$ 

Moreover $\mathrm{depth}(K[\cP])\geq \min\{\mathrm{depth}(K[\cQ_r]), |V(\cP)|-|\cP|\}$. 
\label{lemmaQr}
\end{lemma}

\begin{proof}
Let $\prec$ be the monomial order in $S_\cP$ defined by \cite[Algorithm 4.1]{Cisto_Navarra_CM_closed_path}, that provides a quadratic Gr\"obner basis for $I_\cP$. For all $i\in [r]$, we focus on the initial ideals of $I_{\cQ_i}$ and $I_{\cQ_i'}$ as explained in Remark~\ref{rem:initialQi}. In particular, let $i\in [r]$ and consider the following exact sequence:
$$0 \longrightarrow \frac{S_{\cQ_{i-1}}}{(\mathrm{in}(I_{\cQ_{i-1}}):x_{m_{i}})} \longrightarrow \frac{S_{\cQ_{i-1}}}{\mathrm{in}(I_{\cQ_{i-1}})} \longrightarrow \frac{S_{\cQ_{i-1}}}{(\mathrm{in}(I_{\cQ_{i-1}}),x_{m_{i}})} \longrightarrow0  $$
Arguing as in the proof of \cite[Proposition 4.3]{Cisto_Navarra_Hilbert_series}, we can prove that:
\begin{itemize}
\item $\frac{S_{\cQ_{i-1}}}{(\mathrm{in}(I_{\cQ_{i-1}}),x_{m_{i}})}\cong \frac{S_{\cQ_i}}{\mathrm{in}(I_{\cQ_i})}$

\item $\frac{S_{\cQ_{i-1}}}{(\mathrm{in}(I_{\cQ_{i-1}}):x_{m_{i}})}\cong \frac{S_{\cQ'_{i}}}{\mathrm{in}(I_{\cQ'_{i}})}\otimes K[x_{m_{i}}]$.
\end{itemize}
So by \cite[Corollary 3.3.15]{Villareal} and \cite[Lemmas 5.1.2 and 5.1.3]{Villareal} it follows that for all $i\in \{1\,\ldots,r\}$ it is verified:
$$\mathrm{HP}_{K[\cQ_{i-1}]}(t)=\mathrm{HP}_{K[\cQ_i]}(t)+\frac{t}{1-t}\mathrm{HP}_{K[\cQ'_i]}(t)$$ 

\noindent and also, by Propositions~\ref{conca}, \cite[Lemma 3.1.34]{Villareal} and \cite[Proposition 1.2.9]{Bruns_Herzog}, we obtain $\mathrm{depth}(K[\cQ_{i-1}])\geq \min\{\mathrm{depth}(K[\cQ_{i}]),\mathrm{depth}(K[\cQ_{i}'])+1\}$. In particular, we obtain the first equality of $\rHP_{K[\cP]}(t)$.

Moreover, in the framework of Set-up~\ref{setup1}, we have $|V(\cQ_{i-1})|-|\cQ_{i-1}|=|V(\cQ_{i})|-|\cQ_{i}|$ and $|V(\cQ_{i-1})|-|\cQ_{i-1}|=|V(\cQ'_{i})|-|\cQ'_{i}|+1$ for all $i\in \{1,\ldots,r\}$. So $|V(\cP)|-|\cP|=|V(\cQ'_{i})|-|\cQ'_{i}|+1$ for all $i\in \{1,\ldots,r\}$. Therefore, considering that $\cQ'_i$ is a simple collection of cells, then $K[\cQ'_i]$ is a normal Cohen-Macaulay domain of dimension $|V(\cQ'_{i})|-|\cQ'_{i}|$ for each $i\in [r]$ (by Proposition~\ref{prop:simple_collection-cells-areCM}), in particular $\operatorname{depth}(K[\cQ_i'])=|V(\cQ'_{i})|-|\cQ'_{i}|$. Hence, $\mathrm{depth}(K[\cQ_{i-1}])\geq \min\{\mathrm{depth}(K[\cQ_{i}]),|V(\cP)|-|\cP|\}$, obtaining the bound on $\operatorname{depth}(K[\cP])$. Finally, the second equality in the expression of $\rHP_{K[\cP](t)}$, follows by Corollary~\ref{cor:hilbert-series-thin-collections}.
\end{proof}

\begin{coro}\label{cor:no_skew}
Let $\cP$ be a non-prime closed path polyomino not containing the configuration in Figure~\ref{fig:Lshape}. Then $K[\cP]$ is Cohen-Macaulay of dimension $|V(\cP)|-|\cP|$. 
\end{coro}

\begin{proof}
If $\cP$ is a non-prime closed path polyomino not containing the configuration in Figure~\ref{fig:Lshape} then $\cQ_r$ is a zig-zag collection and by Theorem~\ref{zig-collection} we have that $K[\cQ_r]$ is Cohen-Macaulay of dimension $|V(\cQ_r)|-|\cQ_r|=|V(\cP)|-|\cP|$. Hence $\mathrm{depth}(K[\cP])\geq |V(\cP)|-|\cP|$ but we also know that $\mathrm{dim}(K[\cP])= |V(\cP)|-|\cP|$ (see \cite[Theorem 3.4]{Dinu_Navarra_Konig}), so $K[\cP]$ is Cohen-Macaulay of dimension $|V(\cP)|-|\cP|$.  
\end{proof}

Observe that the collection of cells $\cQ_r$ in Lemma~\ref{lemmaQr} does not contain any configuration as in Figure~\ref{fig:threestepladder}. Now we focus on such a kind of configuration. Having in mind Corollary~\ref{cor:no_skew}, we have to consider the case $\cQ_r$ contains configurations as in Figure~\ref{fig:Lshape}.

\begin{figure}[h]
	\subfloat[]{\includegraphics[scale=0.7]{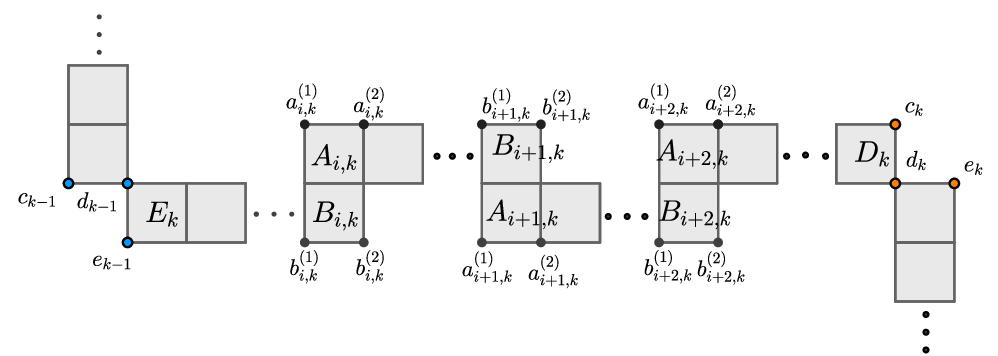}\label{img4_prime-1}}\quad 
	\subfloat[]{\includegraphics[scale=0.7]{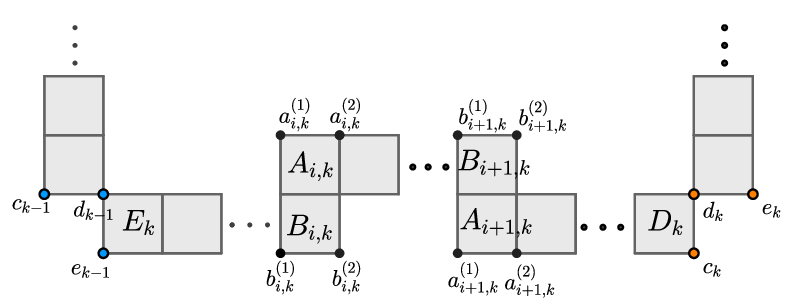}\label{img4_prime-2}}
	\qquad
	\subfloat[]{\includegraphics[scale=0.7]{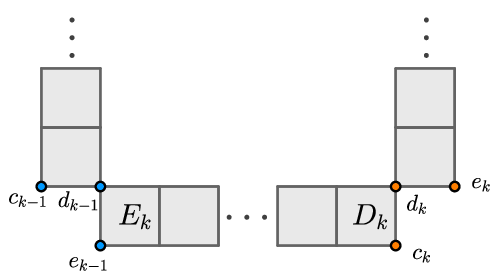}\label{img4_prime-3}}
	\caption{Configurations examined in Discussion~\ref{discussion2} and Set-up~\ref{setup2}} 
	\label{img4_prime}
\end{figure}

\begin{discussion}\label{discussion2} \rm
Let $\cQ_r$ be a collection of cells obtained as in Set-up~\ref{setup1}. Recall that $\cQ_r$ is obtained from a closed path $\cP$ having zig-zag walks, by removing all cells $X$ in each sub-configurations of $\cP$ as in Figure~\ref{fig:threestepladder} (up to reflections and rotations). Observe that, since $\cP$ has not $L$-configurations and ladders of at least three steps, then $\cQ_r$ can be built as a union of opportune reflections and rotations of the configurations in Figure~\ref{img4_prime}. That is, we can express $\cQ_r=\bigcup_{k=1}^{q} \cC_k$, where $\cC_k$ is a configuration as in Figure~\ref{img4_prime-1} or Figure~\ref{img4_prime-2} or Figure~\ref{img4_prime-3} (up to rotations and reflections), for all $k\in \{1,\ldots,q\}$, for some positive integer $q$. In particular, with reference to the points highlighted in the figures, we build $\cQ_r$ by gluing those configurations with this procedure: setting the first configuration $\cC_k$, the second one, say $\cC_{l}$ with $l=k+1$, is glued considering an opportune reflection or rotation of it, in such a way that the vertices $c_k, d_k$ and $e_k$ in orange of the first configuration are overlapped, respectively, to the vertices $c_{l-1},d_{l-1}$ and $e_{l-1}$ in blue of the second configuration. See Figure~\ref{fig:glued} for an example of this construction.
\end{discussion}

\begin{figure}
\includegraphics[scale=0.8]{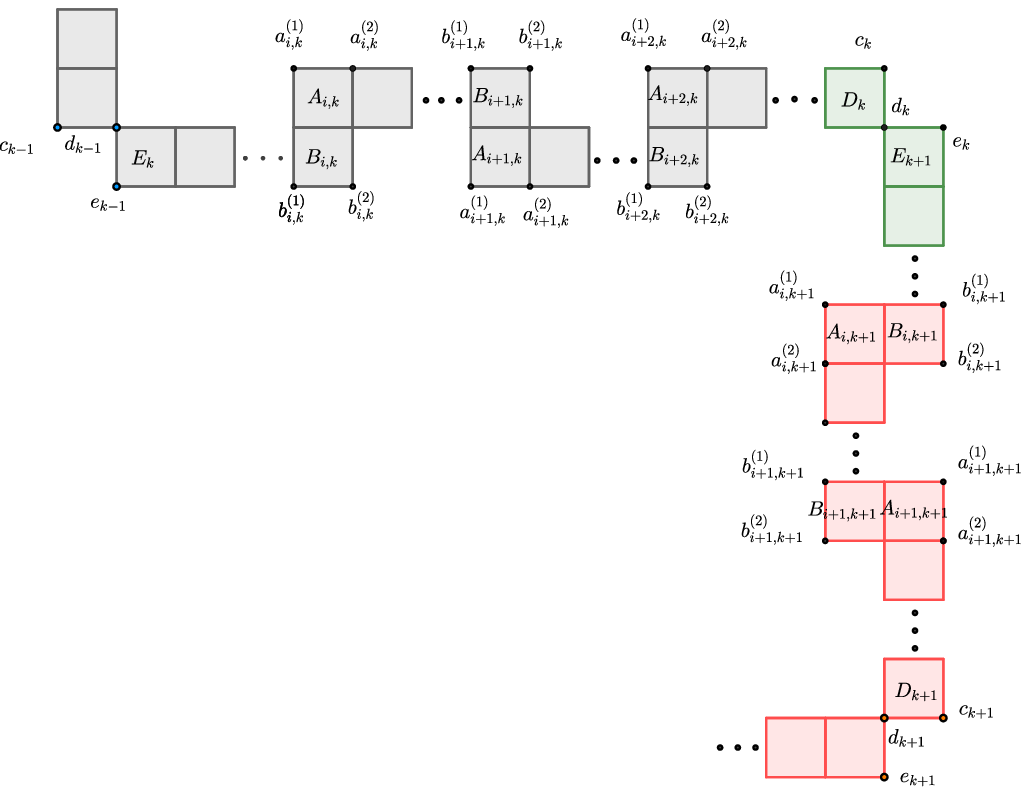}
\caption{Example of gluing of the configurations in Figure~\ref{img4_prime-1} and Figure~\ref{img4_prime-2}, where the overlapped part is in green.}
\label{fig:glued}
\end{figure}

\begin{stp}\label{setup2}\rm
Let $\cQ_r$ be a collection of cells obtained as in Set-up~\ref{setup1}. As explained in Discussion~\ref{discussion2},
set $\cQ_r=\bigcup_{k=1}^{q} \cC_k$ for some positive integer $q$, where $\cC_k$ is a configuration as in Figure~\ref{img4_prime-1} or Figure~\ref{img4_prime-2} or Figure~\ref{img4_prime-3} (up to rotations and reflections), for all $k\in [q]$. Without loss of generality, we can assume that $\cC_1$ is a configuration as in Figure~\ref{img4_prime-1} or Figure~\ref{img4_prime-2} and that has the same orientation provided in those figures. For all $k\in [q]$, denote by $E_k,B_{i,k},A_{i,k},D_k$ and by $c_{k-1},d_{k-1},e_{k-1}, c_k,d_k,e_k, a_{i,k}^{(j_i)}, b_{i,k}^{(j_i)}$, for $j_i\in\{1,2\}$, respectively the cells and the vertices in $\cC_k$ as in Figure~\ref{img4_prime}. 
Moreover, we set $n_k$ as the index such that $a_{n_k,k}^{(1)}$ and $a_{n_k,k}^{(2)}$ belong to the same edge interval of $c_k$. By the structure of $\cQ_r$ we have $c_q=c_0$, $d_q=d_0$ and $e_q=e_0$, where $e_0,d_0,c_0$ are the vertices of $\cC_1$ in blue.

\medskip
Let $J_{\cQ_r}=\{k\in [q]\mid \cC_k\ \mbox{is different to Figure~\ref{img4_prime-3}}\}$. In particular we can assume $J_{\cQ_r}\neq \emptyset$.

For all $k\in J_{\cQ_r}$ and for all $i\in [n_k]$ define:
\begin{enumerate}
\item $\cR_{i,k}=\cQ_r\setminus (\{B_{j,h}\mid h<k, j\in [n_h]\} \cup \{B_{j,k}\mid j\leq i\})$.
\item $\cR'_{i,k}=\cR_{i,k}\setminus \{C\in \cQ_r\mid C\ \mbox{belongs to a maximal cell interval containing}\ B_{i,k}\}.$
\end{enumerate}
\end{stp}

In what follows, we often refer to the framework explained in Set-up~\ref{setup2}, and we made extensive use of the notation introduced there, also with the help of figures.

\begin{exa} \rm

We provide here an example of the framework described in in Set-up~\ref{setup2}, with the help of a figure. In Figure \ref{Figure: some R and R_i'} we show an example of a collection of cells of kind $\cQ_r$ and some of the other related collection of cells. In this case $\cQ_r=\bigcup_{k=1}^{q} \cC_k$ with $q=8$ and $J_{\cQ_r}=\{1,3,7\}$. Observe that the last collection of cells that we build from the described procedure is $\cR_{2,7}$, and it is a zig-zag collection.

 \begin{figure}[h!]
    \centering	
    \subfloat[$\cQ_r$]{\includegraphics[scale=0.50]{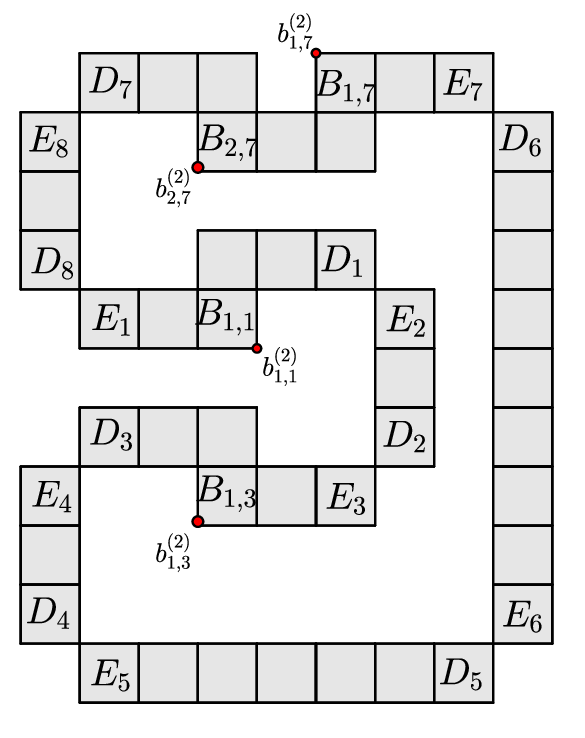}}\qquad
    \subfloat[$\cR_{2,7}$]{\includegraphics[scale=0.50]{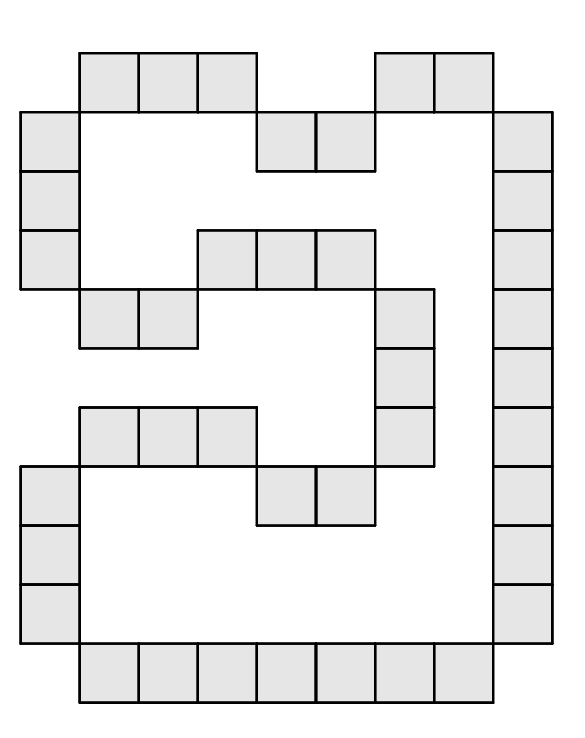}}\qquad
    \subfloat[$\cR_{1,3}'$]{\includegraphics[scale=0.50]{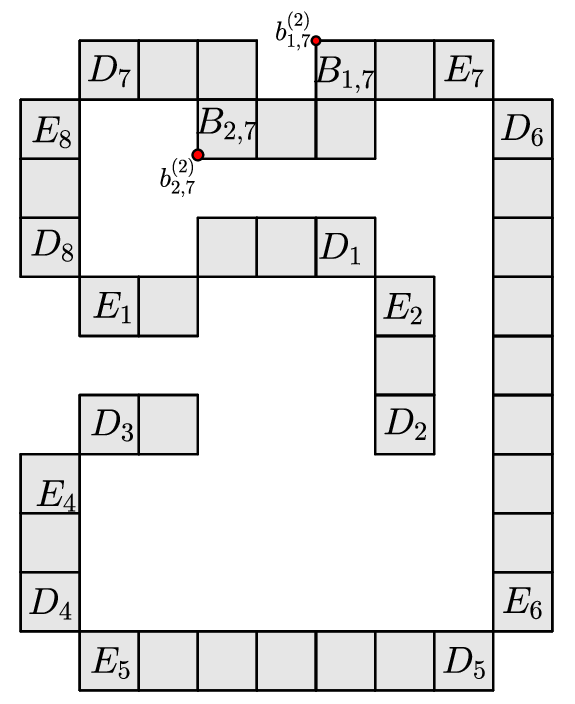}}
    \caption{An example of a collection of cells $\cQ_r$ and the related collections $\cR_{2,7}$ and $\cR_{1,3}'$.}
    \label{Figure: some R and R_i'}
    \end{figure} 
\end{exa}

With reference to Set-up~\ref{setup2}, for the following lemma, observe that if $\prec$ is a monomial order on $S_{\cQ_r}$ then it can be viewed also as a monomial order in $S_{\cR_{i,k}}$ and $S_{\cR_{i,k}'}$ for all $k\in J_{\cQ_r}$ and $i\in [n_k]$.

\begin{lemma}\label{lem:order-Qr}
Let $\cQ_r$ be a collection of cells obtained as in Set-up~\ref{setup1}. Using notations of Set-up~\ref{setup2}, then there exists a term order $\prec$ on $S_{\cQ_r}$ satisfying the following:
\begin{enumerate}

    \item The ideals $I_{\cQ_r}, I_{\cR_{i,k}}$ and $I_{\cR_{i,k}'}$, for all $k\in J_{\cQ_r}$ and for all $i\in [n_k]$, have a squarefree quadratic Gr\"obner basis with respect to $\prec$.

    \item For all $k\in J_{\cQ_r}$ and for all $i\in [n_k]$, if $x_{b_{i,k}^{(2)}}\in \operatorname{supp}(x_u x_v)$ for some inner 2-minor $f=\pm (x_u x_v- x_w x_z)\in I_{\cQ_r}$, then $\operatorname{in}(f)=x_u x_v$.
\end{enumerate}
\end{lemma}

\begin{proof}
Denote $s=\max J_{\cQ_r}$. Observe that $\cR_{n_s,s}$ is a zig-zag collection, so we denote by $I_1,\ldots, I_p$ the set of maximal intervals of such a zig-zag collection, for some integer $p$. We can assume, without loss of generality, that $I_1$ is related to the cell interval $[E_1,B_{1,1}]\setminus\{B_{1,1}\}$, with reference to the notations explained in Set-up~\ref{setup2} and to Figures~\ref{img4_prime-1} and \ref{img4_prime-2}, assuming also that $I_1$ have the same orientation of these figures.

\noindent By Lemma~\ref{grobner}, we know that there exists a term order $\prec_{\cR_{n_s,s}}$ for which  $I_{\cR_{n_s,s}}$ has a quadratic Gr\"obner basis and satisfies:
\begin{itemize}
    \item[(C1)] If $i\in [p]$ is odd, the initial monomials of all inner 2-minors related to $I_i$ are the monomials related to anti-diagonal corners.
    
    \item[(C2)] If $i\in [p]$ is even, the initial monomials of all inner 2-minors related to $I_i$ are the monomials related to diagonal corners.
\end{itemize}

\noindent Now, we focus in $\cQ_r$ on the subconfiguration $\cC_k$ for $k\in [q]$, with reference to Set-up~\ref{setup2}. Recall that each configuration $\cC_k$ is obtained up to reflection and rotations of the configurations in Figure~\ref{img4_prime}. Let $\cC_k'=\cC_k \setminus \{B_{j,k}\mid j\in \{1,\ldots,n_k\}\}$. Therefore, $\cR_{n_s,s}=\bigcup_{k=1}^q \cC_k'$ and the intervals $I_j$ are, up to reflections or rotations, of kind $[E_k, B_{i,k}[$ or $[A_{i,k}, B_{i+1,k}[$ or $[A_{n_k,k},D_k]$, as in Figure~\ref{img4_prime-1}. For all $k\in [q]$, considering the order $\prec_{R_{n_s,s}}$, in the configuration $\cC_k'$ the following holds:

\begin{itemize}
    \item[(R1)] Assume $k\in J_{\cR_{n_s,s}}$. For all $i\in [n_k]$, if $x_{b_{i,k}^{(1)}}\in \operatorname{supp}(x_u x_v)$ with $f=\pm (x_u x_v - x_w x_z)\in I_{\cR_{n_s,s}}$ then $\operatorname{in}(f)=x_u x_v$.

    \item[(R2)] Consider the vertex $d_{k-1}$ with the convention $d_{0}=d_q$. If $x_{d_{k-1}}\in \operatorname{supp}(x_u x_v)$  with $f=\pm (x_u x_v - x_w x_z)\in I_{\cR_{n_s,s}}$ such that $u,v,w,z\in V([E_k,B_{1,k}])$ (or  $u,v,w,z\in V([E_k, D_k])$ if $k\notin J_{\cQ_r}$), then $\operatorname{in}(f)=x_u x_v$.
    
    \item[(R3)] Consider the vertex $d_k$. Then $x_{d_k}\notin \operatorname{supp}(\operatorname{in}(f))$ for every inner interval $f=x_u x_v - x_w x_z\in I_{\cR_{n_s,s}}$ such that $u,v,w,z\in V([A_{n_k,k},D_k])$ (or $u,v,w,z\in V([E_k,D_k])$ if $k\notin J_{\cQ_r}$).

    \item[(R4)] Assume $k\in J_{\cQ_r}$. Set $V_{1,k}=V([E_{k},B_{1,k}]\setminus \{B_{1,k})\})$, $V_{i,k}=V([A_{i-1,k},B_{i,k}]\setminus \{B_{i,k})\})$ for all $i \in [n_k]\setminus \{1,n_k\}$, and $V_{n_k,k}=V([A_{n_k,k},D_{k}])$. For all $i\in [n_k-1]$, set $V_{i,k}\cap V_{i+1,k}=\{h_{i,k}\}$.
    By the construction of $\prec_{\cR_{n_s,s}}$ (see in particular Remark~\ref{rem:order-zig-zag}), for all $u\in V_{i,k}$ and for all $v\in V_{i+1,k}\setminus \{h_{i,k}\}$ we have $x_u\prec_{\cR_{n_s,s}} x_v$. 
\end{itemize}

\noindent Furthermore, by the proof of Lemma~\ref{grobner}, the order $\prec_{\cR_{n_s,s}}$ can be defined as the lexicographic order induced by a total on the variables $x_a$ with $a\in V(\cR_{n_s,s})$. Now we want to define a monomial order $\prec_{\cQ_r}$ on $S_{\cQ_r}$ that extends the order $\prec_{\cR_{n_s,s}}$ defined in $S_{\cR_{n_s,s}}$. Let $Y=\{b_{i,k}^{(2)}\in V(\cQ_r)\mid k\in J_{\cQ_r}, i\in [n_k]\}$ and observe that $V(\cQ_r)=Y\cup V(\cR_{n_s,s})$ (disjoint union). Set $<_Y$ to be any total order on the set $Y$. Let $<_{\cQ_r}$ be the total order on the variables of $S_{\cQ_r}$ defined in the following way:
\[
x_a <_{\cQ_r} x_b \Longleftrightarrow \left \lbrace \begin{array}{l}
    a,b \in V(\cR_{n_s,s})\text{ and } x_a\prec_{\cR_{n_s,s}} x_b \\
    a\in V(\cQ_r)\text{ and }b\in Y \\
    a,b \in Y\text{ and }a<_Y b 
\end{array}
\right.
\]

\noindent Set $\prec_{\cQ_r}$ the lexicographic order on $S_{\cQ_r}$ induced by the total order $<_{\cQ_r}$. Trivially, the monomial order $\prec_{\cQ_r}$ satisfies claim (2). We also want to prove that claim (1) holds. Firstly, we prove it for the ideal $I_{\cQ_r}$. 

\noindent Let $f,g$ be inner 2-minors of $I_{\cQ_r}$ and denote by $S(f,g)$ the $S$-polynomial of $f$ and $g$. It is not difficult to see that if $f,g\in I_{\cR_{n_s,s}}$, since $S(f,g)$ reduces to $0$ in $I_{\cR_{n_s,s}}$, then $S(f,g)$ reduces to $0$ also in $I_{\cQ_r}$. So, suppose $f\notin I_{\cR_{n_s,s}}$.
Hence the inner interval related to $f$ contains the cell $B_{i,k}$ for some $k\in J_{\cQ_r}$ and $i\in [n_k]$. Then, referring to Figure~\ref{fig_lemma-crucial}, up to reflection and rotations, $f$ is one of the following:  
\begin{itemize}
    \item[(a)] $f_{j,i,k}=\pm (x_{b_{i,k}^{(2)}}x_{h_j}-x_{v_j}x_{h_0})$ for $j\in [\ell]$ with $\operatorname{in}(f_{j,i,k})=x_{b_{i,k}^{(2)}}x_{h_j}$.
    \item[(b)] $f_{b,i,k}=\pm (x_{b_{i,k}^{(2)}}x_{a_{i,k}^{(1)}}-x_{a_{i,k}^{(2)}}x_{b_{i,k}^{(1)}})$ with $\operatorname{in}(f_{b,i,k})=x_{b_{i,k}^{(2)}}x_{a_{i,k}^{(1)}}$.
\end{itemize}

\begin{figure}[h]
	\subfloat[]{\includegraphics[scale=0.8]{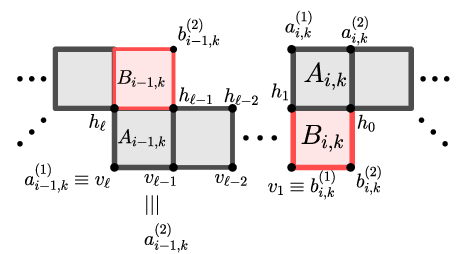}\label{fig_lemma-crucial1}}\quad 
	\qquad
	\subfloat[]{\includegraphics[scale=0.8]{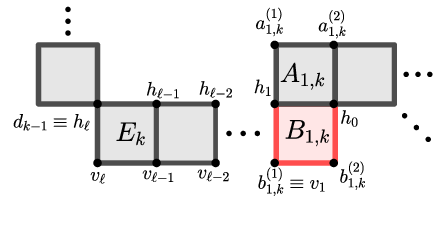}\label{fig_lemma-crucial2}}
	\caption{In the figures above $\ell\geq 3$ and the cells $B_{i-1,k}$, $B_{i,k}$ and $B_{1,k}$ above belong to $\cQ_r$ but not to $\cR_{n_s,s}$} 
	\label{fig_lemma-crucial}
\end{figure}

\noindent For the polynomial $g$, we need to examine only the case $\gcd(\operatorname{in}(f),\operatorname{in}(g))\neq 1$. Then it suffices to examine the following cases:
\begin{itemize}
    \item[Case 1)] $f=f_{j,i,k}$ with $j\in [\ell]$ and $g=\pm (x_{v_{j_1}} x_{h_{j_2}}-x_{v_{j_2}} x_{h_{j_1}})$ with $j_1,j_2\in [\ell]$ or $g=f_{l,i,k}$ with $l \in [\ell]\setminus \{j\}$. In this case $S(f,g)$ reduces to $0$ by \cite[Remark 1]{Cisto_Navarra_CM_closed_path}. 

    \item[Case 2)] $f=f_{1,i,k}$ and $g=f_{b,i,k}$. In this case also, $S(f,g)$ reduces to $0$ by \cite[Remark 1]{Cisto_Navarra_CM_closed_path}.

    \item[Case 3)] $f=f_{j,i,k}$ with $j\in [\ell]\setminus \{1\}$ and $g=f_{b,i,k}$. In this case, by \cite[Lemma 2 (condition 1)]{Cisto_Navarra_CM_closed_path}, we obtain that $S(f,g)$ reduces to $0$ if both the following properties holds:
    \begin{itemize}
        \item[3a)] $x_{a_{i,k}^{(1)}}, x_{h_0}, x_{v_j}\prec_{\cQ_r} x_{a_{i,k}^{(2)}}$.
        
        \item[3b)] $x_{h_1}, x_{v_j}\prec_{\cQ_r} x_{h_j}$ or $x_{h_1}, x_{v_j}\prec_{\cQ_r} x_{v_1}$. 
       
    \end{itemize}
    We prove 3a). Consider that, by observation (R4), $x_{v_j}\prec_{\cQ_r} x_{a_{i,k}^{(2)}}$, $x_{h_1}\prec_{\cQ_r} x_{a_{i,k}^{(1)}}$ and $x_{h_1}\prec x_{h_0}$. Moreover, considering $g'=\pm(x_{a_{i,k}^{(2)}} x_{h_1}-x_{a_{i,k}^{(1)}} x_{h_0})$, then $\operatorname{in}(g')=x_{a_{i,k}^{(2)}} x_{h_1}$ by observation (R1) and conditions (C1)-(C2) of the monomial order $\prec_{\cR_{n_s,s}}$.
    
    \medskip
    \noindent Assuming that $x_{a_{i,k}^{(2)}}\prec_{\cQ_r} x_{a_{i,k}^{(1)}}$ or $x_{a_{i,k}^{(2)}}\prec_{\cQ_r} x_{h_0}$, we obtain $\operatorname{in}(g')=x_{a_{i,k}^{(1)}} x_{h_0}$, that is a contradiction. Hence $x_{a_{i,k}^{(1)}}\prec_{\cQ_r} x_{a_{i,k}^{(2)}}$ and $x_{h_0}\prec_{\cQ_r} x_{a_{i,k}^{(2)}}$. So 3a) is proved.
    
    \noindent Let us prove that 3b) holds. Set $g'=\pm(x_{v_1}x_{h_j}-x_{v_j}x_{h_1})$ and observe that $\operatorname{in}(g')=x_{v_1}x_{h_j}$, by observation (R1) provided at the beginning. If $x_{h_j}\prec_{\cQ_r} x_{v_j}$ and $x_{v_1}\prec_{\cQ_r} x_{v_j}$ then $\operatorname{in}(g')=x_{v_j}x_{h_1}$, a contradiction. We obtain the same if $x_{h_j}\prec_{\cQ_r} x_{h_1}$ and $x_{v_1}\prec_{\cQ_r} x_{h_1}$. Assume $x_{h_j}\prec_{\cQ_r} x_{v_j}$ and $x_{v_1}\prec_{\cQ_r} x_{h_1}$, then comparing $x_{h_j}$ and $x_{v_1}$ with respect to $\prec_{\cQ_r}$, we obtain that $x_{h_j},x_{v_1}\prec_{\cQ_r} x_{v_j}$ or $x_{h_j},x_{v_1}\prec_{\cQ_r} x_{h_1}$, that is, $\operatorname{in}(g')=x_{v_j}x_{h_1}$, a contradiction. We obtain the same if $x_{h_j}\prec_{\cQ_r} x_{h_1}$ and $x_{v_1}\prec_{\cQ_r} x_{v_j}$. So, property 3b) holds.
    
    \item[Case 4)] $f=f_{1,i,k}$ and $g=\pm(x_{h_1}x_u-x_{a_{i,k}^{(1)}} x_w)$ with $u,w\in V([A_{i,k}, B_{i+1,k}]$ if $i\neq n_k$, or $u,w\in V([A_{i,k}, D_k]$ if $i=n_k$, and $\operatorname{in}(g)=x_{h_1}x_u$. In this case, recall that $x_z\prec_{\cQ_r} x_{b_{i,k}^{(2)}}$ for all $z\notin Y$. Hence, if $u\notin Y$ or $u\in Y$ with $x_u\prec_{\cQ_r} x_{b_{i,k}^{(2)}}$ then $S(f,g)$ reduces to $0$ by \cite[Lemma 4 (condition 3)]{Cisto_Navarra_CM_closed_path}, while if $u\in Y$ and $x_{b_{i,k}^{(2)}}\prec_{\cQ_r} x_u$ then $S(f,g)$ reduces to $0$ by \cite[Lemma 4 (condition 1)]{Cisto_Navarra_CM_closed_path}.

    \item[Case 5)] $f=f_{\ell,i,k}$ and $g=\pm(x_{h_\ell}x_{b_{i-1,k}^{(2)}}-x_{b_{i-1,k}^{(1)}} x_{h_{\ell-1}})$. In this case observe that $\operatorname{in}(g)=x_{h_\ell}x_{b_{i-1,k}^{(2)}}$ and we can argue that $S(f,g)$ reduces to $0$ using the same argument of Case 4).
\end{itemize}

\noindent The previous cases allow us to conclude that $I_{\cQ_r}$ has a quadratic Gr\"obner basis with respect to $\prec_{\cQ_r}$. It is not difficult to see that the same arguments hold also to prove that $I_{\cR_{i,k}}$ and $I_{\cR_{i,k}'}$ have a squarefree quadratic Gr\"obner basis with respect to $\prec_{\cQ_r}$, for all $k\in J_{\cQ_r}$ and for all $i\in [n_k]$. So we can conclude.
\end{proof}
    
\noindent The following Lemma, which we need later, is formulated as a general result on a collection of cells. For this reason, only for this result, we use an independent notation, referring to Figure~\ref{fig:lemma-base-conf}.

\begin{lemma}\label{lem:base-configuration}
Let $\cP$ be a collection of cells having a subconfiguration as in Figure~\ref{fig:lemma-base-conf}.
\begin{figure}[h!]
    \centering
    \includegraphics[scale=0.8]{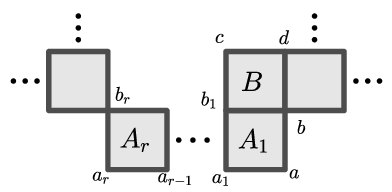}
    \caption{Configuration examined in Lemma~\ref{lem:base-configuration}.}
    \label{fig:lemma-base-conf}
\end{figure}
With reference to that figure, denote $\cQ=\cP \setminus \{A_1\}$ and $\cR=\cP\setminus \{A_1,\ldots, A_r, B\}$ and assume $\{a,c,a_1,\ldots,a_r,b_1,\ldots,b_{r-1}\}\cap V(\cR)=\emptyset$. Suppose that there exists a term order $\prec$ such that $I_\cP$, $I_\cQ$ and $I_\cR$ have a quadratic Gr\"obner basis with respect to $\prec$ and assume $x_a x_c, x_a x_{b_i}\in \operatorname{in}(I_\cP)$ for all $i\in [r]$. Then
\begin{enumerate}
    \item $\rHP_{K[\cP]}(t)=\rHP_{K[\cQ]}(t)+\frac{t}{(1-t)^r}\rHP_{K[\cR]}(t)$.
    \item $\operatorname{depth}(K[\cP])\geq \min \{\operatorname{depth}(K[\cQ]),\operatorname{depth}(K[\cR])+r\}$
\end{enumerate}
Furthermore, if $\cR$ is simple then:
\begin{itemize}
    \item  $\rHP_{K[\cP]}(t)=\rHP_{K[\cQ]}(t)+\frac{t\cdot h_{K[\cR]}(t)}{(1-t)^{|V(\cP)|-|\cP|}}$. 

    \item $\operatorname{depth}(K[\cP])\geq \min \{\operatorname{depth}(K[\cQ]),|V(\cP)|-|\cP|\}$.
\end{itemize}

\end{lemma}
\begin{proof}
   Consider the following exact sequence:

$$0 \longrightarrow \frac{S_{\cP}}{(\mathrm{in}(I_{\cP}):x_a)} \longrightarrow \frac{S_{\cP}}{\mathrm{in}(I_{\cP})} \longrightarrow \frac{S_{\cP}}{(\mathrm{in}(I_{\cP}),x_a)} \longrightarrow0  $$ 

Since $I_\cP,I_\cQ$ and $I_\cR$ have quadratic Gr\"obner basis with respect to the same term order $\prec$ and considering our assumptions, we have:
\begin{itemize}
    \item $\operatorname{in}(I_\cP)=\operatorname{in}(I_\cQ)+(x_a x_c)+ \sum_{i=1}^r (x_a x_{b_i})$.
    \item $\operatorname{in}(I_\cP)=\operatorname{in}(I_\cR)+(x_a x_c)+ \sum_{i=1}^r (x_a x_{b_i}) +\max_\prec (\{x_{a_i} x_{b_j}, x_{a_j} x_{b_i} \mid x_{a_i} x_{b_j}-x_{a_j} x_{b_i} \in I_\cP, i,j\in [r], i<j\})+ \max_\prec (\{x_{c} x_{z_j}, x_{b_1} x_{w_j} \mid x_{c} x_{z_j}-x_{b_1} x_{w_j} \in I_\cP\}).$
\end{itemize}
As a consequence, we obtain:

\begin{itemize}
    \item $(\operatorname{in}(I_\cP),x_a)=(\operatorname{in}(I_\cQ),x_a)$ and in particular $\frac{S_\cP}{(\operatorname{in}(I_\cP),x_a)}\cong \frac{S_\cQ}{\operatorname{in}(I_\cQ)}$.

    \item $\operatorname{in}(I_\cP):x_a= \operatorname{in}(I_\cR)+ (x_c)+\sum_{i=1}^r(b_i)$, and in particular $\frac{S_{\cP}}{(\mathrm{in}(I_{\cP}):x_a)} \cong \frac{S_\cR}{(\operatorname{in}(I_\cR), x_{b_r})}\otimes K[x_a, x_{a_1},\ldots,x_{a_r}]$ .
\end{itemize}
Observe that, since by assumptions $I_\cP$ has a quadratic Gr\"obner basis with respect to $\prec$, $x_{b_r}$ does not divide any generator of $\operatorname{in}(I_\cP)$ different by $x_{a_i} x_{b_r}$, for $i\in [r-1]$. In fact, assume $f=x_{b_r}x_v-x_z x_w\in I_\cP$ for some $v,w,z\in V(\cP)$ such that $v\in V(\cP)\setminus \{a,a_1,\ldots,a_r\}$ and $\operatorname{in}(f)=x_{b_r}x_v$. So, considering $g=x_a x_{b_r}-x_b x_{a_r}\in I_\cP$, we have $\gcd (\operatorname{in}(f),\operatorname{in}(g))\neq 0$. Without loss of generality, assume $z$ is in the same horizontal edge interval of $b_r$, while $w$ is in the same vertical edge interval of $b_r$. As a consequence, assuming that the $S$-polynomial of $f$ and $g$ reduces to zero, by \cite[Lemma 6]{Cisto_Navarra_CM_closed_path} we have $[z,a_r]$ or $[w,a]$ is an inner interval, a contradiction. 

\noindent Therefore, if $y$ is another indeterminate and denoting by $S'$ the polynomial ring such that $S'\otimes K[x_{b_r}]\cong S_\cR$, we can express $$\frac{S_\cR}{\operatorname{in}(I_\cR)}\cong \frac{S'}{\operatorname{in}(I_\cR)}\otimes K[x_{b_r}]\cong \frac{S_\cR}{(\operatorname{in}(I_\cR), x_{b_r})}\otimes K[y] $$

\noindent Hence, by \cite[Lemma 5.1.11]{Villareal}, we have $\rHP_{S_\cR/(\operatorname{in}(I_\cR), x_{b_r})}(t)=(1-t)\cdot \rHP_{S_\cR/\operatorname{in}(I_\cR)}(t)$. Therefore, by \cite[Corollary 3.3.15]{Villareal} and \cite[Lemmas 5.1.2 and 5.1.3]{Villareal} we obtain $\rHP_{K[\cP]}(t)=\rHP_{K[\cR]}(t)+\frac{t}{(1-t)^r}\rHP_{K[\cR]}(t)$, that is the first claim.

\noindent For the second claim, by  \cite[Proposition 1.2.9]{Bruns_Herzog} we have $\operatorname{depth}(S_\cP/\operatorname{in}(I_\cP))\geq \{\operatorname{depth}(S_\cQ/\operatorname{in}(I_\cQ)),\operatorname{depth}(S_\cR/(\operatorname{in}(I_\cR), x_{b_r})\otimes K[x_a, x_{a_1},\ldots,x_{a_r}])\}$. Hence, by Propositions~\ref{conca} and \cite[Lemma 3.1.34]{Villareal} we obtain $\operatorname{depth}(K[\cP])\geq \min \{\operatorname{depth}(K[\cQ]),\operatorname{depth}(K[\cR])+r\}$. 

Finally, for the last two claims, if $\cR$ is a simple collection of cells, by Proposition~\ref{prop:simple_collection-cells-areCM} we have that $\dim (K[\cR])=\operatorname{depth}(K[\cR])=|V(\cR)|-|\cR|$. Moreover $|V(\cR)|=|V(\cP)|-(2r+1)$ and $|\cR|=|\cP|-(r+1)$, that is, $|V(\cR)|-|\cR|+r= |V(\cP)|-|\cP|$. From this fact we can easily conclude. 
\end{proof}

\begin{lemma}
Let $\cP$ be a non-prime closed path polyomino and $\cQ_r$ be the collection of cells as in Set-up~\ref{setup1}. Consider the framework of Set-up~\ref{setup2}. Then $\mathrm{depth}(K[\cQ_r])\geq |V(\cQ_r)|-|\cQ_r|$. 

\noindent Moreover, if $s=\max J_{\cQ_r}$, then: 
$$\mathrm{HP}_{K[\cQ_r]}(t)=\mathrm{HP}_{K[\cR_{n_s,s}]}(t)+\frac{t}{(1-t)^{|V(\cQ_r)|-|\cQ_r|}}\left(\sum_{k\in J}\sum_{i=1}^{n_k}h_{K[\cR'_{i,k}]}(t)\right)$$ 
\label{lemmaRs}
\end{lemma}

\begin{proof}
    Let $k\in J_{\cQ_r}$ and $i\in [n_k]$ and denote $h=b_{i,k}^{(2)}$. Consider, respectively, the following collection of cells and cell interval:

    \[
\cS_{i,k} :=
\left\{
\begin{array}{ll}
\cR_{i-1,k} & \mbox{if $i\neq 1$} \\
\cR_{n_{k-1},k-1} & \mbox{if $i=1$ and $k\neq \min J_{\cQ_r}$} \\
\cQ_r & \mbox{otherwise}
\end{array}
\right.
\qquad \mbox{and}  \qquad \mathcal{E}_{i,k}:=\left\{
\begin{array}{ll}
[A_{i-1,k};B_{i,k}] & \mbox{if}\ i\neq 1 \\ 
\left[ E_k;B_{1,k}\right] & \mbox{if $i=1$} 
\end{array}
\right.
\]

Observe that, up to rotations and reflections, the configuration of cells given by $\mathcal{E}_{i,k}\cup \{A_{i,k}\}$ determine a sub-configuration of $\cS_{i,k}$ as in Figure~\ref{fig:lemma-base-conf}. Moreover, observe that $\cS_{i,k}\setminus \{B_{i,k}\}=\cR_{i,k}$ and $\cS_{i,k}\setminus (\mathcal{E}_{i,k}\cup \{A_{i,k}\})=\cR_{i,k}'$. Furthermore, by Lemma~\ref{lem:order-Qr}, we have that $\cS_{i,k}$ satisfies all conditions of Lemma~\ref{lem:base-configuration} and $\cR_{i,k}'$ is a simple collection of cells. Therefore $\rHP_{K[\cS_{i,k}]}(t)=\rHP_{K[\cR_{i,k}]}(t)+\frac{t\cdot h_{K[\cR'_{i,k}]}(t)}{(1-t)^{|V(\cS_{i,k})|-|\cS_{i,k}|}}$. and $\operatorname{depth}(K[\cS_{i,k}])\geq \min \{\operatorname{depth}(K[\cR_{i,k}]),|V(\cS_{i,k})|-|\cS_{i,k}|\}$. So, denote $J_{\cQ_r}=\{k_1,\ldots, k_s\}$. Starting from $\cQ_r$ we obtain:

\begin{align*} 
\rHP_{K[\cQ_r]}(t) & =\rHP_{K[\cR_{1,k_1}]}(t)+\frac{t \cdot h_{K[\cR_{1,k_1}']}(t)}{(1-t)^{|V(\cQ_r)-|\cQ_r|}}=\\
& =\rHP_{K[\cR_{2,k_1}]}(t)+\frac{t}{(1-t)^{|V(\cQ_r)-|\cQ_r|}}\left(h_{K[\cR_{1,k_1}']}(t)+h_{K[\cR_{2,k_1}']}(t)\right) =\\
& = \cdots = \\
& = \mathrm{HP}_{K[\cR_{n_{k_s},k_s}]}(t)+\frac{t}{(1-t)^{|V(\cQ_r)|-|\cQ_r|}}\left(\sum_{k\in J_{\cQ_r}}\sum_{i=1}^{n_k}h_{K[\cR'_{i,k}]}(t)\right)
\end{align*}
Furthermore,
\begin{align*}
\operatorname{depth}(K[\cQ_r])\geq & \min \{\operatorname{depth}(K[\cR_{1,k_1}]),|V(Q_r)|-|\cQ_r|\}. \\
\operatorname{depth}(K[\cR_{1,k_1}])\geq & \min \{\operatorname{depth}(K[\cR_{2,k_1}]),|V(\cR_{1,k_1})|-|\cR_{1,k_1}|\}.\\
\cdots & \\
\operatorname{depth}(K[\cS_{n_{k_s}, k_s}])\geq & \min \{\operatorname{depth}(K[\cR_{n_{k_s}, k_s}]),|V(\cS_{n_{k_s}, k_s})|-|\cS_{n_{k_s}, k_s}|\}.
\end{align*}

\noindent Observing that $|V(\cQ_r)|-|\cQ_r|=|V(\cR_{i,k})|-|\cR_{i,k}|$ for all $k\in J_{\cQ_r}$ and for all $i\in [n_k]$ (in fact, $\cS_{i,k}$ differs from $\cR_{i,k}$ only for a vertex and cells), we obtain $\operatorname{depth}(K[\cQ_r])\geq \min \{\operatorname{depth}(K[\cR_{n_{k_s}, k_s}]),|V(Q_r)|-|\cQ_r|\}$. 
Finally, observe that $\cR_{n_{k_s}, k_s}$ is a zig-zag collection. Hence, by Theorem~\ref{zig-collection}, we have that $K[\cR_{n_{k_s}, k_s}]$ is a Cohen-Macaulay domain of dimension $|V(\cR_{n_{k_s}, k_s})|-|\cR_{n_{k_s}, k_s}|=|V(\cQ_r)|-|\cQ_r|$, obtaining $\operatorname{depth}(K[\cQ_r])\geq |V(Q_r)|-|\cQ_r|$

\end{proof}

    \begin{thm}\label{Thm: CM + rook for closed path with zig-zag}
    Let $\cP$ be a closed path with a zig-zag walk. Then $K[\cP]$ is Cohen-Macaulay of dimension $|V(\cP)|-|\cP|$. 
    \end{thm}

\begin{proof}
     By Lemma~\ref{lemmaQr} we have $\mathrm{depth}(K[\cP])\geq \min\{\mathrm{depth}(K[\cQ_r]), |V(\cP)|-|\cP|\}$ and by Lemma~\ref{lemmaRs} $\mathrm{depth}(K[\cQ_r])\geq |V(\cQ_r)|-|\cQ_r|$. Moreover, by the proof of Lemma~\ref{lemmaQr} we also have $|V(\cQ_r)|-|\cQ_r|=|V(\cP)|-|\cP|$. Hence, $\operatorname{depth}(K[\cP])\geq |V(\cP)|-|\cP|$. By \cite[Theorem 3.4]{Dinu_Navarra_Konig} we also know that $\dim (K[\cP])=|V(\cP)|-|\cP|$, so $K[\cP]$ is Cohen-Macaulay with the given dimension.
\end{proof}

\begin{rmk} \label{rem:CM-conf}\rm
        Observe that the result in Theorem~\ref{Thm: CM + rook for closed path with zig-zag} holds also for all collection of cells $\cQ_i$ and $\cR_{i,k}$ described in Set-up~\ref{setup1} and Set-up~\ref{setup2}. In fact, if $\cR$ is one of such configurations, by \cite[Lemma 3.3]{Cisto_Navarra_closed_path} and by construction, we can argue that there exists a cell $A$ in $\cR$ such that $\cR\setminus \{A\}$ is simple. So, by \cite[Proposition 4.9]{Cisto_Navarra_Veer} (and \cite[Corollary 3.1.7]{Villareal}) we have $\dim(K[\cR])=|V(\cR)|-|\cR|$. Therefore, we can obtain this claim using the same arguments developed in this section.
    \end{rmk}

 \begin{thm}\label{Thm:rook for closed path with zig-zag}
    Let $\cP$ be a closed path with a zig-zag walk. Then the $h$-polynomial of $K[\cP]$ is the rook polynomial of $\cP$ and $\mathrm{reg}(K[\cP])=r(\cP)$.
    \end{thm}
    \begin{proof}
          We start by showing how the claim holds for the collection of cells $\cQ_r$ as in Set-up~\ref{setup1}. If $\cP$ is a non-prime closed path polyomino not containing the configuration in Figure~\ref{fig:Lshape} then $\cQ_r$ is a zig-zag collection, so $h_{K[\cQ_r]}(t)$ is the rook polynomial by Corollary~\ref{Coro: For zig-zag coll, h is the switching rook pol}. Assume $\cP$ contains the configuration in Figure~\ref{fig:Lshape}. Then, with reference to Set-up~\ref{setup2} and the proof of Lemma~\ref{lemmaRs}, for $k\in J_{\cQ_r}$ and $i\in [n_k]$ consider the collections of cells $\cS_{i,k}$ and $\cR_{i,k}'$. In Lemma~\ref{lemmaRs} we proved that
           $$\rHP_{K[\cS_{i,k}]}(t)=\rHP_{K[\cR_{i,k}]}(t)+\frac{t\cdot h_{K[\cR'_{i,k}]}(t)}{(1-t)^{|V(\cS_{i,k})|-|\cS_{i,k}|}}.$$

         \noindent Moreover, for all $k\in J_{\cQ_r}$ and for all $i\in [n_k]$, observe that $\cR_{i,k}'$ is a simple collection of cells . So, by Corollary~\ref{cor:hilbert-series-thin-collections}, $h_{K[\cR'_{i,k}]}(t)$ is the rook polynomial of $\cR'_{i,k}$. Furthermore, for all $k\in J_{\cQ_r}$ and for all $i\in [n_k]$, by the definition of $\cS_{i,k}$ and by the proof of Lemma~\ref{lemmaRs} we have $|V(\cS_{i,k})|-|\cS_{i,k}|=|V(\cR_{i,k})|-|\cR_{i,k}|=|V(\cQ_r)|-|\cQ_r|$, and by Remark~\ref{rem:CM-conf} we have $\dim(K[\cR_{i,k}])=|V(\cR_{i,k})|-|\cR_{i,k}|=|V(\cQ_r)|-|\cQ_r|$. As a consequence, we have:

         $$\rHP_{K[\cS_{i,k}]}(t)=\frac{h_{K[\cR_{i,k}]}(t)+t\cdot h_{K[\cR'_{i,k}]}(t)}{(1-t)^{|V(\cQ_r)|-|\cQ_r|}}.$$

         \noindent Let $s_1=\max J_{\cQ_r}$ and $s_2=\min J_{\cQ_r}$. Following the collections of cells involved, observe that $\cR_{n_s,s}$ is a zig-zag collection and and $\cS_{1,s_1}=\cQ_r$. In particular, by Corollary~\ref{Coro: For zig-zag coll, h is the switching rook pol}, $h_{K[\cR_{n_s,s}]}(t)$ is the rook polynomial of $\cR_{n_s,s}$. Therefore, if $k\in J_{\cQ_r}$ and $i\in [n_k]$, it suffices to prove the following claim: if $h_{K[\cR_{i,k}]}(t)$ is the rook polynomial of $\cR_{i,k}$ then $h_{K[\cS_{i,k}]}(t)$ is the rook polynomial of $\cS_{i,k}$. 
         
         \noindent So, assume $h_{K[\cR_{i,k}]}(t)$ is the rook polynomial of $\cR_{i,k}$. For all non-negative integer $k$, denote by $r^{(1)}_k$ and $r^{(2)}_k$, respectively, the coefficient of $t^k$ in the polynomials $h_{K[\cR_{i,k}]}(t)$ and $h_{K[\cR_{i,k}']}(t)$. Denote by $r_k=r^{(1)}_k+r^{(2)}_{k-1}$ (with the convention $r_{-1}^{(2)}=0$). Observe that $r_k$ is the coefficient of $t^k$ in $h_{K[\cS_{i,k}]}(t)$. Hence, it suffices to show that $r_k$ is the number of $k$-rook configurations of $\cS_{i,k}$. Recall that, as explained in the proof of Lemma~\ref{lemmaRs}, we can refer to  Figure~\ref{fig:lemma-base-conf} and express $\cR_{i,k}=\cS_{i,k}\setminus \{A_1\}$ and $\cR_{i,k}'=\cS_{i,k}\setminus ( \{B,A_1,\ldots,A_r\})$. So, consider that we can write $r_k=r_1+r_2$, where $r_1$ is the number of $k$-rook configurations of $\cS_{i,k}$ having no rook placed in $A_1$ and $r_2$ is the number of $k$-rook configurations of $\cS_{i,k}$ with a rook placed in $A_1$. It is not difficult to argue that $r_1$ is also equal to the number of $k$-rook configurations of $\cR_{i,k}$, while $r_2$ is equal the number of $(k-1)$-rook configurations of $\cR_{i,k}'$. Therefore, $r_1=r_k^{(1)}$ and $r_2=r_{k-1}^{(2)}$, obtaining that $r_k$ is the number of $k$-rook configurations of $\cS_{i,k}$. As a consequence, $h_{K[\cR_{i,k}]}(t)$ is the rook polynomial of $\cS_{i,k}$, proving our claim.

          \noindent So, the $h$-polynomial of $K[\cQ_r]$ is the rook polynomial of $\cQ_r$. In order to obtain our results for $\cP$, let us consider, with reference to Set-up~\ref{setup1}, the collections of cells $\cQ_i$ and $\cQ_i'$ for all $i\in [r]$. In Lemma~\ref{lemmaQr} we showed that $|V(\cP)|-|\cP|=|V(\cQ_i)|-|\cQ_i|=|V(\cQ_i')|-|\cQ_i'|+1$ for all $i\in [r]$ and we also proved the following:
          
          $$\mathrm{HP}_{K[\cQ_{i-1}]}(t)=\mathrm{HP}_{K[\cQ_i]}(t)+\frac{t}{1-t}\mathrm{HP}_{K[\cQ'_i]}(t)$$ 

          \noindent Moreover, for all $i\in [r]$ consider that $\cQ_i'$ is a simple collection of cells, in particular $\dim(K[\cQ_i'])=|V(\cQ_i')|-|\cQ_i'|$ and by Corollary~\ref{cor:hilbert-series-thin-collections} we have that $h_{K[\cQ_i']}$ is the rook polynomial of $\cQ_i'$. So, considering also Remark~\ref{rem:CM-conf}, by the previous expression we obtain:

          $$\mathrm{HP}_{K[\cQ_{i-1}]}(t)=\frac{h_{K[\cQ_i]}(t)+t\cdot h_{K[\cQ_i']}}{(1-t)^{|V(\cP)|-|\cP|}}$$

          \noindent Following the collections of cells involved and observing that $\cQ_0=\cP$, we can obtain our result if for all $i\in [r]$ we prove the following claim: if $h_{K[\cQ_i]}(t)$ is the rook polynomial of $\cQ_i$ then $h_{K[\cQ_{i-1}]}(t)$ is the rook polynomial of $\cQ_{i-1}$. However, it is not difficult to see that this fact can be proved by the same argument used before, so we can conclude.
    \end{proof}

\begin{rmk} \rm
    It is not difficult to see that the claim provided for $\cP$ in Theorem~\ref{Thm:rook for closed path with zig-zag}, following its proof, holds also for all collection of cells $\cQ_i$ and $\cR_{i,k}$ described in Set-up~\ref{setup1} and Set-up~\ref{setup2}.
\end{rmk}
    
\begin{prop}\label{Prop: closed path with zig-zag are not Gore}
    Let $\cP$ be a closed path with a zig-zag walk. Then $K[\cP]$ is not Gorenstein.
\end{prop}

\begin{proof}
     Denote by $h(t)=\sum_{k=1}^{s} h_kt^k$ the $h$-polynomial of $K[\cP]$. From Theorem~\ref{Thm:rook for closed path with zig-zag}, $h(t)$ is the rook polynomial of $\cP$ and $s=r(\cP)$. Suppose by contradiction that $K[\cP]$ is Gorenstein, so it follows from \cite[Corollary 5.3.10]{Villareal} that $h_k=h_{s-k}$ for all $k\in \{0,\dots,s\}$. In particular we have that $h_{s-1}=h_1=\vert \cP\vert$ and also $h_s=h_0=1$. We examine the following three cases. 
     \begin{enumerate}
      \item  Let us start showing that $\vert I\vert \leq 3$ for every maximal cell interval $I$ of $\cP$. In fact, suppose there exists a maximal interval $I = [A, B]$ with $\vert I\vert \geq 4$. Since $\vert I\vert \geq 4$, then we can consider two distinct cells $C, D \in I\setminus\{A, B\}$. Note that every $s$-rook configuration of $\cP$ has a rook in a cell of $I\setminus \{A,B\}$. Let $\cT$ be an $s$-rook configuration of $\cP$ with a rook placed in $C$. Therefore, we obtain a new $s$-rook configuration of $\cP$, moving the rook in $C$ into $D$, so $h_s \geq 2 > h_0 = 1$, that is a contradiction. Hence, $\vert I\vert \leq 3$ for every maximal interval $I$ of $\cP$, as claimed.
      \item Now, suppose that $\cP$ contains a sub-polyomino as in Figure \ref{fig:Lshape}. Arguing as done before, there exists an $s$-rook configuration $\mathcal{U}$ in $\cP$ having a rook placed in $W$ and another one in $Z$. Moreover, we can get from $\mathcal{U}$ another $s$-rook configuration in $\cP$, moving the rook in $\mathcal{U}$ from $W$ to $X$. Therefore the same previous contradiction arises. 
      \item From the previous cases it follows that every maximal interval of $\cP$ has a rank less than or equal to $3$ and $\cP$ does not contain any sub-polyomino as in Figure \ref{fig:Lshape}. Hence the only possibility is that $\cP$ is as in Figure \ref{Figure: the smallest closed path with zig-zag}. In such a case, \texttt{Macaulay2} (\cite{Package_M2,M2}) computations show that $h_{K[\cP]}(t)=1+16\,t+100\,t^{2}+308\,t^{3}+486\,t^{4}+376\,t^{5}+134\,t^{6}+20\,t
     ^{7}+t^{8},$ so $h_7=20\neq 16=h_1$, which is a contradiction. 
      \end{enumerate}

\noindent In all three cases we get a contradiction, so we conclude that $K[\cP]$ cannot be Gorenstein.
\end{proof}

In the following result we gather, as a summary, the main algebraic properties of $K[\cP]$ when $\cP$ is a closed path polyomino.

\begin{thm}\label{thm:summary-closed-path}
Let $\cP$ be a closed path polyomino. Then:
\begin{enumerate}
    \item  $K[\cP]$ is Cohen-Macaulay with Krull dimension $|V(\cP)|-|\cP|$; moreover, if $\cP$ does not contain any zig-zag walk, then $K[\cP]$ is also a normal domain. 
    \item  The $h$-polyonomial of $K[\cP]$ is the rook polynomial of $\cP$ and $\mathrm{reg}(K[\cP])=r(\cP)$.
 \item $K[\cP]$ is Gorenstein if and only if $\cP$ consists of maximal blocks of rank three. 
\end{enumerate}
\end{thm}

\begin{proof}
    We need to discuss just two cases. In the first one assume that $\cP$ has a zig-zag walk. Then (1) and (2) follow from Theorems \ref{Thm: CM + rook for closed path with zig-zag} and \ref{Thm:rook for closed path with zig-zag}. Concerning (3), we know that $K[\cP]$ is not Gorenstein from Proposition \ref{Prop: closed path with zig-zag are not Gore}. Moreover, by \cite[Proposition 6.1]{Cisto_Navarra_closed_path} $\cP$ does not contain any $L$-configuration, so any two maximal inner intervals of $\cP$ intersect themselves in the cells displayed in Figure~\ref{fig:threestepladder} or~\ref{fig:Lshape} (up to reflections or rotations). Hence, we argue that if $\cP$ has a zig-zag walk then it contains maximal blocks of length two. This means that (3) holds if $\cP$ has zig-zag walks.  Now, in the second case suppose that $\cP$ does not contain a zig-zag walk. The claims (1) and (2) are already proved in \cite[Theorem 5.5]{Cisto_Navarra_Hilbert_series} and (3) in \cite[Theorem 5.7]{Cisto_Navarra_Hilbert_series}. In conclusion, the theorem is completely proved. 
\end{proof}
 
    We proved that the zig-zag collections and the closed paths with zig-zag walks are Cohen-Macaulay, as well as the collections of cells $\cQ_i$ and $\cR_{i,k}$ introduced respectively in Set-up~\ref{setup1} and Set-up~\ref{setup2}. Nowadays, an example of a collection of cells having a not Cohen-Macaulay coordinate ring is still unknown. Therefore, from \cite{Simple equivalent balanced, def balanced, Simple are prime} and from the results of this work, the following general question naturally arises. 
\begin{qst}
 Let $\cP$ be a collection of cells. Then, is $K[\cP]$ Cohen-Macaulay?   
\end{qst}

\begin{small}
    \textbf{Acknowledgement.} The third author is supported by Scientific and Technological Research Council of Turkey T\"UB\.{I}TAK under the Grant No: 122F128, and he is thankful to T\"UB\.{I}TAK for their support. The first and the third author acknowledge support of INDAM-GNSAGA. The second author is supported by Scientific and Technological Research Council of Turkey T\"UB\.{I}TAK under the Grant No: 118F169, and he is thankful to T\"UB\.{I}TAK for their support. The second author also acknowledges the support of Erasmus funding during his stay at the University of Messina. Special thanks to Professor Rosanna Utano for her support and hospitality during this period.
\end{small}

\end{document}